\documentclass[12pt,letter]{amsart}
\usepackage[text={6.5in,9in},centering]{geometry}

\usepackage{hyperref}
\newcommand{\utilde}{\tilde}

\usepackage{amsmath,amssymb}
\usepackage{latexsym}
\usepackage{mathrsfs}
\usepackage{amsthm}
\usepackage{graphicx}
\usepackage{color}
\usepackage{colortbl}
\usepackage[config,labelfont={rm,rm},textfont=rm]{caption,subfig}

\newtheorem{theorem}{Theorem}[section]

\newtheorem{lemma}[theorem]{Lemma}
\newtheorem{proposition}[theorem]{Proposition}

\newtheorem{definition}[theorem]{Definition}

\newtheorem{remark}[theorem]{Remark}

\usepackage[latin1]{inputenc}
\usepackage[english]{babel}

\newcommand{\calQh}{ {\mathcal Q}_h}
\newcommand{\calQho}{ {\mathcal Q}_h}

\newcommand{\grad}{{\boldsymbol \nabla}}
\newcommand{\gradh}{{\boldsymbol{\mathcal G}}_h}
\renewcommand{\div}{\textrm{div}}
\newcommand{\curl}{\textbf{curl}}
\newcommand{\bdiv}{\textbf{div}}
\newcommand{\ccurl}{\textrm{curl}}

\newcommand{\Qk}{\mathcal{S}(\K)}
\newcommand{\Nb}{\vect{\mathcal{N}}_h}

\newcommand{\Nbo}{\stackrel{o}\Nb}
\newcommand{\calX}{{\vect{\mathcal{X}}}_h}
\newcommand{\Mk}{\vect{\mathcal{M}}(\K)}
\newcommand{\Rk}{\vect{\mathcal{R}}(\K)}
\newcommand{\Pp}{ \mathbb{P}}
\newcommand{\Pb}{\mathbb{ \vect{P}}}
\newcommand{\cMk}{\mathcal{M}(\K)}
\newcommand{\bMk}{ \vect{\mathcal{M}}(\K)}
\newcommand{\BDM}{{\bf BDM}}
\newcommand{\RT}{{\bf RT}}
\newcommand{\TR}{{\bf TR}}
\newcommand{\BDFM}{{\bf BDFM}}
\newcommand{\ND}{{\bf ND}}

\newcommand{\Velh}{\vect{V}_h}
\newcommand{\Velho}{\vect{V}_h}
\newcommand{\Velhoo}{\mathring{\vect{V}}_h}
\newcommand{\Nh}{\mathcal{N}_h}
\newcommand{\Nho}{\mathcal{N}_h}

\newcommand{\calWh}{ {\mathcal W}_h}
\newcommand{\calWho}{\stackrel{o}\calWh}
\newcommand{\R}{{\mathbb R}}

\newcommand{\Gsym}[1]{\vect{\varepsilon}(\vect{#1})}

\newcommand{\K}{T} 
\renewcommand{\O}{\Omega} 
\newcommand{\Eh}{{{\mathcal E}_h}} 
\newcommand{\Eho}{{{\mathcal E}^{o}_h}} 

\newcommand{\Ehb}{{{\mathcal E}^{\partial}_h}} 
\newcommand{\Th}{\mathcal{T}_h} 
\newcommand{\hK}{\h_{\K}} 
\newcommand{\h}{h} 

\newcommand{\bx}{{\bf x}}

\newcommand{\bu}{\vect{u}}

\newcommand{\f}{\vect{f}}

\newcommand{\vp}{\vect{\psi}}
\newcommand{\sig}{\vect{\sigma}}

\newcommand{\bv}{\vect{v}}
\newcommand{\bw}{\vect{w}}
\newcommand{\bvh}{\vect{v}_h}

\newcommand{\taub}{{\boldsymbol {\tau}}}

\newcommand{\av}[1]{\{#1\}} 

\newcommand{\n}{{\bf n}} 
\newcommand{\nK}{{\bf n}_{\K}}
\newcommand{\dx}{\,\mbox{d}x} 
\newcommand{\ds}{\,\mbox{d}s} 
\newcommand{\jump}[1]{\lbrack\!\lbrack\,#1\,\rbrack\!\rbrack} 

\newcommand{\trin}{|\!|\!|}
\newcommand{\re}{ \mathbb{R}}

\numberwithin{equation}{section}
\numberwithin{figure}{section}
\numberwithin{table}{section}

\renewcommand\lll{|\kern-1pt|\kern-1pt|} 

\newcommand{\salto}[1]{\left[\!\left[ #1 \right]\!\right]}

\newcommand{\vect}[1]{\boldsymbol{#1}}

\renewcommand{\lor }{\longrightarrow}

\newcommand{\p}{ \partial}

\newcommand{\triplenorm}[1]{%
  \left\vert\kern-0.9pt\left\vert\kern-0.9pt\left\vert #1
  \right\vert\kern-0.9pt\right\vert\kern-0.9pt\right\vert}
\begin{document}

\title[DG for Stokes]{A simple preconditioner for a discontinuous Galerkin method for the Stokes problem}

\author[B. Ayuso de Dios]{Blanca Ayuso de Dios}
\address{Centre de Recerca Matem\'{a}tica, UAB Science Faculty, 08193 Bellaterra, Barcelona, Spain}

\author[F. Brezzi]{Franco Brezzi}
 \address{IUSS-Pavia c/p IMATI-CNR, Via Ferrata 5/A 27100 Pavia Italy, and
Dept. of Math., KAU, PO Box 80203, Jeddah, 21589, Saudi Arabia}
\author[L. D. Marini]{L. Donatella Marini}
\address{Dipartimento di Matematica, Universit\`a di Pavia,
Via Ferrata 1, 27100 Pavia, Italy}

\author[J. Xu]{Jinchao Xu}\address{Department of Mathematics, Penn State
University, University Park PA 16802, USA}

\author[L. Zikatanov]{Ludmil Zikatanov}\address{Department of
  Mathematics, Penn State University, University Park PA 16802, USA}


\begin{abstract}
  In this paper we construct Discontinuous Galerkin approximations of
  the Stokes problem where the velocity field is
  $H(\div,\O)$-conforming. This implies that the velocity solution is
  divergence-free in the whole domain. This property can be exploited
  to design a simple and effective preconditioner for the final linear
  system.
\end{abstract}

\maketitle


\section{Introduction}\label{sec:0}

In this paper we present a preconditioning strategy for a family of
discontinuous Galerkin discretizations of the Stokes problem in a
domain $\Omega \subset \re^{d}, d=2,3$:
\begin{equation}\label{Stokes:A}
 \left\lbrace{
 \begin{aligned}
  -{ \bdiv} (2\nu \Gsym{\bu}) + \grad p ~&=~{\bf f}\quad \mbox{in }\O \\
  {\bf \div} \, \bu &=~0 \quad \mbox{in }\O
  \end{aligned}
  } \right.
  \end{equation}
where, with the usual notation, $\vect{u}$ is the velocity field, $p$
the pressure, $\nu$ the viscosity of the fluid, and $\Gsym{u}\in
[L^{2}(\O)]^{d\times d}_{{\rm sym}}$ is the symmetric (linearized) strain rate tensor defined by  $\Gsym{u}=\frac{1}{2} ( \grad \bu + (\grad \bu)^{T})$.

The methods considered here were introduced in \cite{WangYe} for the
Stokes problem and in \cite{dominik0} for the Navier-Stokes equations
when pure Dirichlet boundary conditions are prescribed.  In both
works, the authors showed that the approximate velocity field is
exactly divergence-free, namely it is $H(\div;\Omega)$-conforming and
divergence-free almost everywhere.  These same methods were also used
in \cite{domink:mhd0}.

Numerical methods that perserve divergence free condition exactly are
important from both practical and theoretical points of view. First of
all, it means that the numerical method conserves the mass everywhere,
namely, for any $D\subset \Omega$ we have
$$
\int_{\partial D}\bu\cdot \vect{n}=0.
$$
As an example of its theoretical importance, the exact divergence free
condition plays a crucial view for the stability of the mathematical
models (see \cite{Lin.F;Liu.C;Zhang.P.2005a}) and their numerical
discretizations (see \cite{Lee.Y;Xu.J2006}) for complex fluids.

The focus of this paper is to develop new solvers for the resulting
algebraic systems for this type of discretization by exploring the
divergence-free property.  In general, the numerical discretization of
the Stokes problem produces algebraic linear systems of equations of
the saddle-point type.  Solving such algebraic linear systems has been
the subject of considerable attention from various communities and
many different approaches can be used to solve them efficiently (see
\cite{ESW2005} and references cited therein). One popular approach is
to use a block diagonal preconditioner with two blocks: one containing
the inverse or a preconditioner of the stiffness matrix of a vector
Poisson discretization, and one containing the inverse of a lumped
mass matrix for the pressure.  This preconditioner when used in
conjunction with MINRES (MINimal RESidual) leads to a solver which is
uniformly convergent with respect to the mesh size.

While the existing solvers such as this diagonal preconditioner can
also be used for these DG methods, in this paper, we would like to
explore an alternative approach by taking the advantage of the
divergence-free property.  Our new approach reduces the solution of
the Stokes systems (which is indefinite) to the solution of several
Poisson equations (which are symmetric positive definite) by using
auxiliary space preconditioning techniques, which we
hope would open new doors for the design of algebraic solvers for PDE
systems that involve subsystems that are related to Stokes operator.

{In \cite{dominik0,WangYe} the classical Stokes operator is considered
  for the special case of purely homogeneous Dirichlet boundary
  conditions (no-slip Dirichlet's condition). While this special case
  is theoretically important, it does not model well most of the
  cases that occur in the engineering applications (for instance, it
  is not realistic in applications in immiscible two-phase flows,
  aeronautics, in weather forecasts or in hemodynamics).} For the pure
homogenous no-slip Dirichlet boundary conditions, we have the
following identity
\[
\int_\Omega \Gsym{u} : \Gsym{v}= \int_\Omega \grad \bu: \grad \bv.
\]
when $\bu$ and $\bv$ vanish on the boundary of $\Omega$.  This
identity can be used when deriving the variational formulation, thus
leading to simplifications of the analysis in the details related to
the Korn's inequality on the discrete level.


To extend the results in \cite{dominik0,WangYe} to this different
boundary condition we provide detailed analysis showing that the
resulting DG-$\vect{H}(\div;\Omega)$-conforming methods are stable and
converge with optimal order.  Furthermore, a key feature of the
DG-$\vect{H}(\div;\Omega)$-conforming schemes of providing a
divergence-free velocity approximation is satisfied as in
\cite{dominik0,WangYe}, by the appropriate choice of the
discretization spaces. This property is fully exploited in designing
and constructing efficient preconditioners and we reduce the solution
of the Stokes problem to the solution of a ``second-order'' problem in
the space $\curl\, H^1_0(\Omega)$.

We propose then a preconditioner for the solution of the corresponding
problem in $\curl \, H^1_0(\Omega)$. This is done by means of the
fictitious space \cite{NEP1991, NEP1992} (or auxiliary space
\cite{JXU96,Oswald96}) framework. The proposed preconditioner amounts
to the solution of one vector and two scalar Laplacians. The solution
of such systems can then be {\it efficiently} computed with classical
approaches, for instance the Geometric Multigrid (GMG) or Algebraic
Multigrid (AMG) methods.

Throughout the paper, we use the standard notation for Sobolev
spaces~\cite{Adams75}. For a bounded domain $D\subset \mathbb{R}^{d}$,
we denote by $H^{m}(D)$ the $L^2$-Sobolev space of order~$m \geq 0$
and by $\|\cdot \|_{m,D}$ and $| \cdot |_{m,D}$ the usual Sobolev norm
and seminorm, respectively. For $m = 0$, we write $L^{2}(D)$ instead
of $H^{0}(D)$. For a general summability index $p$, we also denote by
$W^{m,p}(D)$ the usual $L^p$-Sobolev spaces of order $m\geq 0$ with
norm $\|\cdot \|_{m,p,D}$ and seminorm $| \cdot|_{m,p,D}$.  By
convention, we use boldface type for the vector-valued analogues:
$\vect{H}^{m}(D)=[H^{m}(D)]^{d}$, likewise, we use boldface italics for the
symmetric-tensor-valued analogues: $ \vect{\mathcal
  H}^{m}(D):=[H^{m}(D)]^{d\times d}_{\rm{sym}}$. $H^{m}(D)/\mathbb{R}$ denotes
  the quotient space consisting of equivalence
classes of elements of $H^{m}(D)$ that differ by a constant; for $m=0$
the quotient space is denoted by $L^{2}(D)/\mathbb{R}$. We indicate by $L^{2}_{0}(D)$
the space of the $L^{2}(D)$ functions with zero average over~$D$
(which is obviously isomorphic to $L^{2}(D)/\mathbb{R}$). We use
$(\cdot\,,\cdot)_{D}$ to denote the inner product in the spaces
$L^{2}(D),\vect{L^{2}}(D)$, and $\vect{\mathcal L}^{2}(D)$.


\section{Continuous Problem}
In this section, we discuss the well posedness of the Stokes problem
which is of interest.  We remark that the results in
the paper are valid in two and three dimensions, although to make the
presentation more transparent we focus on the two dimensional case,
discussing only briefly the main changes (if any) needed to carry over the
results to three dimensions.

We begin by restating (for reader's convenience) the equations already
given in~\eqref{Stokes:A} with a bit more detail regarding the
boundary conditions.  For a simply connected polyhedral domain $\Omega
\subset \re^{d}, d=2,3$ with boundary $\Gamma=\partial\O$, we consider
the Stokes equations for a viscous incompressible fluid:
\begin{equation}\label{Stokes:F}
 \left\lbrace{
 \begin{aligned}
  -{ \bdiv} (2\nu \Gsym{\bu}) + \grad p ~&=~{\bf f}\quad \mbox{in }\O \\
  {\bf \div} \, \bu &=~0 \quad \mbox{in }\O
  \end{aligned}
  } \right.
  \end{equation}
On the boundary $\Gamma$ we impose kinematic boundary condition
\begin{equation}\label{bc}
\bu\cdot \n=0 \quad \mbox{ on   } \Gamma,
\end{equation}
together with the natural condition on the tangential component of the normal stresses
\begin{equation}\label{bc00}
 ((2\,\nu\,\Gsym{u} -p \mathbf{I})\n)\cdot \vect{t}=0 \quad \mbox{ on   } \Gamma,
\end{equation}
where $\mathbf{I}$ is the identity tensor. Note that as $\n\cdot \vect{t}\equiv 0$
then \eqref{bc00} is reduced to
\begin{equation}\label{bc0}
 (\Gsym{u}\n)\cdot \vect{t}=0 \quad \mbox{ on   } \Gamma.
\end{equation}

When the space
 \begin{equation}\label{defHn0}
 \vect{H}^{1}_{0,n}(\O)=\{\bv \in \vect{H}^{1}(\O) \,\,\, :\,\,\, \bv \cdot \n=0 \mbox{   on   } \Gamma\,\}
 \end{equation}
  is introduced, the  variational formulation of the Stokes problem  reads: {\it Find $(\bu , p)\in \vect{H}^{1}_{0,n}(\O) \times L^{2}(\O)/\re$ as the solution of:}
\begin{equation}\label{var:0}
\left\{\begin{aligned}
a(\bu,\bv)+b(\bv,p)&=(\f,\bv) \quad &\forall\, \bv \in \vect{H}^{1}_{0,n}(\O) &&\\
b(\bu,q)&=0 \quad\quad &\forall\, q \in L^{2}(\O)/\re &&
\end{aligned}\right.
\end{equation}
where  for all $\bu\in\vect{H}^{1}_{0,n}(\O)$,
$\bv\in\vect{H}^{1}_{0,n}(\O)$ and
$q\in\, L^{2}(\O)/\re$ the (bi)linear forms are defined by
\begin{equation*}
a(\bu,\bv) :=2\nu \int_{\O} \Gsym{u}: \Gsym{v} \dx, \quad
b(\bv,q):=-\int_{\O} q\, \div \,\bv \dx, \quad
(\f,\bv):= \int_{\O}\f\cdot\bv\dx.
\end{equation*}
For the classical mathematical treatment of the Stokes problem (where
the Laplace operator is used instead of the divergence of the stress tensor $\Gsym{u}$) existence and
uniqueness of the solution $(\bu,p)$ are very well known and have been
reported with different boundary conditions in many places (see for
instance \cite{ladyz, temam1, galdi:0,girault-raviart}).  The
Stokes problem considered here
\eqref{Stokes:F}-\eqref{bc}-\eqref{bc00} has been derived and used in
different applications \cite{temam2,beavers, guido-riviere}. 

{For the Stokes problem with the slip boundary conditions
  \eqref{bc}-\eqref{bc00}, existence, uniqueness and interior
  regularity was first established in \cite{solonikov73} (for even the
  more general linearized Navier-Stokes).  The study of well-posedness
  and regularity up to the boundary for the solutions of this problem
  has received substantial attention only in very recent years.  For
  example, analysis can be found in \cite{beirao2004, amrouche2011}
  for weak and strong solutions in the $H^{1}(\O)\times L^{2}(\O)$ and
  $W^{1,p}(\O)\times L^{p}(\O), \, 1<p<\infty$. In these works it is
  assumed that the boundary of $\O$ is at least of class
  $\mathcal{C}^{1,1}(\O)$ and the more general boundary condition of
  Navier slip-type is studied.  In \cite{harbir}, the authors provide
  the analysis in the $W^{1,p}(\O)\times L^{p}(\O), \, 1<p<\infty$ for
  less regular domains.
 
  Here, for the sake of completeness, we provide a very brief outline
  of the proof of well-posedness of the problem, in the case $\O$ is a
  polygonal or polyhedral domain (which is the relevant case for the
  numerical approximation we have in mind).}  By introducing the
operator $D_0=-\div: \vect{H}^{1}_{0,n}(\O) \lor L^{2}_{0}(\O)$, it
can be shown \cite{Brezzi.F.1974a,temam1} that $D_0$ is surjective,
i.e., Range$(D_0)=L^{2}_0(\O)$. Therefore, the operator $D_0$ has a
continuous lifting which implies that the continuous inf-sup condition
is satisfied. Hence, from the classical theory follows that to
guarantee the well-posedness of the Stokes problem
\eqref{Stokes:F}-\eqref{bc}, it is enough to show that the bilinear
form $a(\cdot,\cdot)$ is coercive; ie., there exists $\gamma_0>0$ such
that
\begin{equation}\label{cont:coer}
a(\bv,\bv)\geq \gamma_0 |\bv|_{1,\O}^{2} \qquad \forall \, \bv\in  \vect{H}^{1}_{0,n}(\O).
\end{equation}
Once continuity is established, existence, uniqueness and a-priori estimates follow in a standard way.
The proof of \eqref{cont:coer} requires a Korn inequality, that in general imposes some restrictions on the domain (see Remark \ref{rem:domain00}). For the case considered in this work the needed result  is contained in next Lemma:
 \begin{lemma}\label{korn:cont:new}
Let $\O\subset \re^{d}\;, d=2,3$ be a polygonal or polyhedral domain. Then, there
exists a constant $C_{Kn}>0$ (depending on the domain through its diameter and shape) such that
\begin{equation}\label{kornn0}
|\bv|_{1,\O}^{2} \leq C_{Kn}  \| \Gsym{\bv} \|_{0,\O}^{2}, \quad \forall\, \bv \in  \vect{H}^{1}_{0,n}(\O).
\end{equation}
\end{lemma}
To prove the above Lemma, we first need the following auxiliary result
\begin{lemma}\label{domain}
For every polygonal or polyhedral domain $\O$ there exists a positive constant
$\kappa(\O)$ such that
\begin{equation}\label{tool0}
\kappa(\O)\|\vect{\eta}\|^2_{0,\O}\le\|\vect{\eta}\cdot\vect{n}\|^2_{0,\partial\O}
\qquad\forall\vect{\eta}\in \vect{RM}(\O)
\end{equation}
where $\vect{RM}(\O)$ is the space of rigid motions on $\O$ defined by
\begin{equation*}
\vect{RM}(\O)=\left\{  \vect{a}+\vect{b} \vect{x}\,\,:\,\, \vect{a}\in \re^{d} \quad \vect{b} \in so(d) \right\}
\end{equation*}
with $so(d)$ denoting the  set of skew-symmetric $d\times d$ matrices, $d=2,3$.
\end{lemma}
\begin{proof}
To ease the presentation we provide the proof only in two dimensions. The extension to three dimensions involve only notational changes and therefeore it is ommitted.
To show the lemma we observe that a polygon contains always at least two edges not belonging
to the same straight line. A rigid movement whose normal component vanishes identically on
those two edges is easily seen to be identically zero. This implies that  for $\vect{c}\equiv (c_1,c_2,c_3)\in\R^3$ on the (compact) manyfold
$$\displaystyle{\int_{\O}|(c_1-c_3x_2,c_2+c_3x_1)|^2\dx=1}$$
the function
\begin{equation}
\vect{c}\rightarrow\int_{\partial\O}|(c_1-c_3x_2,c_2+c_3x_1)\cdot\vect{n}|^2\ds
\end{equation}
(which  is obviously continuous) is never equal to zero. Hence it has a positive minimum, that equals the
required  $\kappa(\O)$.
\end{proof}
As a direct consequence of last Lemma, we can now provide the proof of the desired Korn inequality given in Lemma \ref{korn:cont:new}.
\begin{proof}
{\it (Proof of Lemma \ref{korn:cont:new}.)}\\
For every $\bv\in \vect{H}^{1}_{0,n}(\O)$ we consider first its $L^2$ projection
$\bv_R$ on the space $\vect{RM}(\O)$ of rigid motions and the projection $\bv_{\perp}:=
\bv-\bv_R$ on the orthogonal subspace. As $\bv\cdot\vect{n}=0$ on $\partial \O$ we obviously
have
\begin{equation}\label{tool1}
\bv_R\cdot\vect{n}=-\bv_{\perp}\cdot\vect{n}.
\end{equation}
Moreover, as $\bv_{\perp}$ is orthogonal to rigid motions we have
\begin{equation}\label{tool2}
|\bv_{\perp}|_{1,\O}^{2} \leq C_{K}  \| \Gsym{\bv_{\perp}} \|_{0,\O}^{2}
\end{equation}
for some constant $C_K$ (note that the rigid motions include the constants, so that
Poincar\'e inequality also holds for $\bv_{\perp}$).  On the other hand, since $\vect{RM}(\O)$ is finite dimensional
we have obviously
\begin{equation}\label{tool3}
|\bv_{R}|_{1,\O}^{2} \leq C_P  \|\bv_{R} \|_{0,\O}^{2}
\end{equation}
that using \eqref{tool0} gives
\begin{equation}\label{tool4}
|\bv_{R}|_{1,\O}^{2} \leq \frac{C_P}{\kappa(\O)} \|\bv_{R}\cdot\vect{n} \|_{0,\partial \O}^{2}
\end{equation}
and using also \eqref{tool1} and \eqref{tool2}
\begin{equation}\label{tool5}
\begin{aligned}
\frac{1}{2}|\bv|_{1,\O}^{2}&\le|\bv_{R}|_{1,\O}^{2}+|\bv_{\perp}|_{1,\O}^{2}
\leq \frac{C_P}{\kappa(\O)} \|\bv_{R}\cdot\vect{n} \|_{0,\partial\O}^{2}+|\bv_{\perp}|_{1,\O}^{2}\\
&=
\frac{C_P}{\kappa(\O)} \|\bv_{\perp}\cdot\vect{n} \|_{0,\partial \O}^{2}+|\bv_{\perp}|_{1,\O}^{2}\le \frac{C_T\,C_P}{\kappa(\O)} |\bv_{\perp}|_{1,\O}^{2}\\
&\le \frac{C_T\,C_P\,C_K}{\kappa(\O)} \|\Gsym{v_{\perp}}\|_{0,\O}^{2}=
\frac{C_T\,C_P\,C_K}{\kappa(\O)} \|\Gsym{v}\|_{0,\O}^{2}
\end{aligned}
\end{equation}
where the constant $C_T$ depends on the trace inequality on $\O$. Defining now $C_{Kn}=\frac{2C_T\,C_P\,C_K}{\kappa(\O)}$ we conclude the proof.
\end{proof}

\begin{remark}\label{rem:domain00}
  The proof of Lemma \ref{korn:cont:new} relies on the assumption that
  the domain is polygonal or polyhedral. For more general smooth
  bounded domains, the Korn inequality \eqref{kornn0} is still true,
  as long as the domain is assumed to be not rotationally
  symmetric. Otherwise a Korn inequality can be established by
  restricting the solution space (see \cite[Appendix A]{muller_korn}
  for further details).
 \end{remark}

\section{Abstract setting and basic notations}\label{sec:1}

Let $\Th$ be a shape-regular family of partitions of $\O$ into triangles $\K$ in $d=2$ or tetrahedra in $d=3$. We denote by  $\hK$ the diameter of $\K$, and we set $\h=\max_{\K \in \Th} \hK$. We also assume that the decomposition  $\Th$ is conforming in the sense that it does not contain any hanging nodes.

We denote by $\Eh$ the set of all edges/faces and by $\Eho$ and $\Ehb$ the collection of
all interior and boundary edges, respectively.

For $s\geq 1$, we define
\begin{equation*}
H^{s}(\Th)=\left\{ \phi \in
    L^{2}(\O)~\mbox{, such that}~\phi\big|_{\K} \in
    H^{s}(\K), \quad \forall\, \K\in \Th\,\right\},
\end{equation*}
and their vector $\vect H^{s}(\Th)$ and tensor $\vect {\mathcal H}^{s}(\Th)$ analogues, respectively.
For scalar, vector-valued,  and tensor functions, we use  $(\cdot \,,\cdot)_{\Th}$ to denote the $L^{2}(\Th)$-inner product and $\langle \cdot\,, \cdot \rangle_{\Eh}$ to denote the $L^{2}(\Eh)$-inner product elementwise.\\
The vector functions are represented column-wise.
We recall the definitions of the following operators acting on vectors $\vect{v}\in\vect{H}^{1}(\O)$  and on scalar functions $\phi \in H^{1}(\O)$ as
\begin{equation*}
\begin{aligned}
&\div \,\vect{v} =\sum_{i=1}^{d} \frac{\partial v^{i}}{\partial x_i} \qquad \qquad \qquad \qquad \qquad \qquad  &&\\
&\ccurl \,\vect{v} = \frac{\partial v^{2}}{\partial x_1}- \frac{\partial v^{1}}{\partial x_2}\qquad \qquad \curl \,\phi =\nabla^{\perp} \phi:=\left[ \frac{\partial \phi}{\partial x_2}, -\frac{ \partial \phi}{\partial x_1}\right]^{T} \qquad  (d=2) &&\\
&\curl \,\vect{v} =\vect{\nabla} \vect{\times} \vect{v} =\left[ \; \frac{\partial v^{3}}{\partial x_2}- \frac{\partial v^{2}}{\partial x_3} \, , \, \frac{\partial v^{1}}{\partial x_3}-\frac{\partial v^{3}}{\partial x_1}\; ,\, \frac{\partial v^{2}}{\partial x_1} -\frac{\partial v^{1}}{\partial x_2}\; \right]^{T} \qquad (d=3)&&
\end{aligned}
\end{equation*}

And, we recall the definitions of the spaces to be used herein:
\begin{equation*}
\begin{aligned}
\vect{H}(\div;\O)&:=\{ \vect{v} \in \vect{L^{2}}(\O) \,\, :\,\,\, \div \,\vect{v} \in L^{2}(\O)\, \} ,\quad d=2,3,&&\\
\vect{H}(\ccurl;\O) &:=\{  \vect{v} \in \vect{L^{2}}(\O) \,\, :\,\,\, \ccurl \,\vect{v} \in L^{2}(\O)\, \} \quad d=2,&&\\
\vect{H}(\curl;\O) &:=\{ \vect{v} \in \vect{L^{2}}(\O) \,\, :\,\,\, \curl \,\vect{v} \in \vect{L^{2}}(\O)\, \} \quad d=3\;.&&
\end{aligned}
\end{equation*}
\begin{equation*}
\begin{aligned}
\vect{H}_{0,n}(\div;\O)&:=\{ \vect{v} \in \vect{H}(\div;\O)\,:\,\, \vect{v}\cdot \n =0 \,\,\mbox{  on  } \Gamma\,\}, &&\\
 \vect{H}_{0,t}(\curl;\O)&:= \{ \vect{v} \in \vect{H}(\curl;\O) \,:\,\, \vect{v}\vect{\times} \n =0 \,\,\mbox{  on  } \Gamma\,\}, &&\\
\vect{H}_{0,n}(\div^{0};\O)&:=\{ \vect{v} \in \vect{H}_{0,n}(\div;\O)\,:\,\, \div\, \vect{v} =0 \,\,\mbox{  in  } \Omega\,\}. &&
 \end{aligned}
\end{equation*}
The above spaces are Hilbert spaces with the norms
\begin{align*}
\|\bv\|_{H(\div,\O)}^{2}&:=\|\bv\|_{0,\O}^{2}+\|\div \,\bv\|_{0,\O}^{2}\qquad \forall\, \bv \in \vect{H}(\div;\O)\;, &&\\
 \|\bv\|_{H(\ccurl,\O)}^{2}&:=\|\bv\|_{0,\O}^{2}+\|\ccurl \,\bv\|_{0,\O}^{2}\qquad \forall\, \bv \in \vect{H}(\ccurl;\O) . &&\\
   \|\bv\|_{H(\curl,\O)}^{2}&:=\|\bv\|_{0,\O}^{2}+\|\curl \,\bv\|_{0,\O}^{2}\qquad \forall\, \bv \in \vect{H}(\curl;\O)\;. &&
\end{align*}

\begin{remark} \label{sudiv0} It is worth noting that if we restrict our analysis to vectors $\bu$ and $\bv$ in $\vect{H}^1(\O)\cap\vect{H}_{0,n}(\div^{0};\O)$ then problem \eqref{var:0} becomes: {\it Find $\bu\in
\vect{H}^1(\O)\cap\vect{H}_{0,n}(\div^{0};\O)$ as the solution of:}
\begin{equation}\label{var:div0}
a(\bu,\bv)=(\f,\bv) \quad \forall\, \bv \in \vect{H}^1(\O)\cap\vect{H}_{0,n}(\div^{0};\O).
\end{equation}
\end{remark}

As is usual in the DG approach, we now  define some trace operators.
 Let $e\in \Eho$ be an internal edge/face of $\Th$ shared by two elements $\K^{1}$ and $\K^{2}$, and let $\n^{1}$ ($\n^{2}$) denote  the unit normal on $e$ pointing outwards from $\K^{1}$ ($\K^{2}$). For a scalar function $\varphi \in H^{1}(\Th)$, a  vector field $\taub\in \vect{H}^{1}(\Th)$, or a tensor field $\taub\in \vect {{\mathcal H}^{1}}(\Th)$ we define the average operator in the usual way (see for instance \cite{abcm}), that is, on internal edges/faces
\begin{equation*}
\av{\varphi}=\frac {1} {2}(\varphi^1 +\varphi^2),\quad \av{\bv}=\frac {1} {2}(\bv^1 +\bv^2), \quad \av{\taub}=\frac {1} {2}(\taub^1 +\taub^2).
\end{equation*}
However, on a boundary edge/face, we take $\av{\varphi},~ \av{\bv}$, and $\av{\taub}$ as the trace
of $\varphi$, $\bv$, and $\taub$,respectively, on that edge.

For a scalar function $\varphi \in H^{1}(\Th)$, the jump operator is defined as
\begin{equation*}
\jump{\varphi}=\varphi^1\n^1+\varphi^2\n^2  \quad\mbox{on } e\in \Eho, \mbox{ and }
\jump{\varphi}=\varphi\n \quad\mbox{on } e\in \Ehb
\end{equation*}
(where obviously $\n$ is the outward unit normal), so that the jump of a scalar function is a vector in the normal direction.

For a vector field $\bv\in \vect{H}^{1}(\Th)$, following, for example, \cite{abmR},  the jump is the symmetric matrix-valued function given on $e$ by
\begin{equation*}
\salto{\bv}=\bv^1\odot\n^1 + \bv^2\odot\n^2 \quad\mbox{on } e\in \Eho, \mbox{ and }
\salto{\bv}=\bv\odot\n  \quad\mbox{on } e\in \Ehb,
\end{equation*}
where $\bv \odot \n=(\bv \n^{T}+\n \bv^{T})/2$ is the symmetric part
of the tensor product of $\bv$ and $\n$.
Hence, the jump of a vector-valued function is a symmetric tensor.

If we denote by $\nK$ the outward  unit normal to $\partial\K$, it is easy to check that
\begin{equation}\label{per-vect}
\sum_{\K\in \Th} \int_{\p \K} \bv\cdot \nK\, q\, \ds= \sum_{e\in \Eh} \int_{e} \av{\bv} \cdot \jump{q}\, \ds \quad \forall\, \bv \in \vect{H}^{1}(\Th)\;, \,\,\,\forall\, q\in H^{1}(\Th).
\end{equation}
Also for $\taub \in \vect{\mathcal H}^{1}(\O)$ and for all $\bv \in \vect{H}^{1}(\Th)$, we have
\begin{equation}\label{per-tens0}
\sum_{\K\in \Th} \int_{\p \K} (\taub\nK) \cdot \bv \,\ds= \sum_{e\in \Eh} \int_{e} \av{\taub} : \salto{\bv}\, \ds.
\end{equation}

\subsection{Discrete Spaces: General framework}

We present three choices for each of the finite element spaces $\Velh$
and $\calQh$ to approximate velocity and pressure, respectively. For
each choice, we also need an additional space $\Nh$ (resp. $\Nb$ in
$d=3$) made of piecewise polynomial scalars and of piecewise
polynomial vectors in three dimensions, to be used as a sort of {\it
  potentials} or {\it vector potentials}. We will explain the reason
for doing this and the way in which to do this later on. Note, too,
that we will use this space more heavily in the construction of our
preconditioner.  The different choices for the spaces $\Velh$,
$\calQh$, and $\Nh$ or $\Nb$ rely on different choices of the local
polynomial spaces $\Rk$, $\Qk$, and $\cMk$ or $\bMk$, respectively,
made for each element $\K$. Specifically, we have
\begin{equation}\label{vel:0}
\Velh:=\left\{ \bv \in \vect{H}(\div;\O)\,\, :\,\,  \bv |_{\K} \in \Rk\,\,
\forall\, \K\in \Th,~~\bv\cdot\n=0~\mbox{on } \Gamma \right\} ,
\end{equation}
\begin{equation}\label{defQh}
\calQh:=\left\{ q \in L^2(\Omega)/\re\,\, :\,\,  q |_{\K} \in \Qk\,\,
\forall\, \K\in \Th \right\},
\end{equation}
and
\begin{equation}\label{defWh2d}
\Nh:=\left\{\varphi \in H^{1}_0(\O)\,\, :\,\,  \varphi |_{\K} \in \cMk\,\,
\forall\, \K\in \Th \right\} \mbox{ for $d=2$, and }
\end{equation}
\begin{equation}\label{nede:0:3d}
\Nb:=\left\{
\bv \in \vect{H}(\curl ;\O)\,\, :\,\,\,  \bv |_{\K} \in \Mk\,\,
\forall\, \K\in \Th ~~\bv\times \n=0~\mbox{on } \Gamma  \right\} \mbox{ for $ d=3$}.
\end{equation}

The three spaces $\Velho$, $\calQho$, and $\Nho$ (or $\Nb$) will always be related by this
exact sequences:
\begin{equation}\label{seqex2d}
0{\;\longrightarrow\;}\Nho\stackrel{\curl}{\;\longrightarrow\;}\Velho\stackrel{\div}{\;\longrightarrow\;}
\calQho {\;\longrightarrow\;}0 .
\end{equation}
in two dimensions, and
\begin{equation}\label{seqex3d}
0{\;\longrightarrow\;}\Nb \stackrel{\curl}{\;\longrightarrow\;}\Velho\stackrel{\div}{\;\longrightarrow\;}\calQho
\end{equation}
in three dimensions.
 It is also necessary for each operator in \eqref{seqex2d} and \eqref{seqex3d} to have a continuous right inverse whose
norm is uniformly bounded in $h$. For instance, it is necessary that
\begin{equation}\label{divboundinv}
\exists\,\beta>0 \,\mbox{ s.t.}~~\forall h,\forall\, q\in\calQho\, \exists\,\vect{v}\in\Velho\mbox{ with: } \div\,\vect{v}=q\quad\mbox{ and }\quad  \|\vect{v}\|_{0,\O}\le\frac{1}{\beta}\|q\|_{0,\O}.
\end{equation}
Obviously, for the $\curl$ operator (in 2 and 3 dimensions) these bounded right inverses will be defined only
on $\Velho\cap \vect{H}_{0,n}(\div^0,\O)$.

\begin{remark}\label{formixed} In all our examples, the pair $(\Velho,\calQho)$ is among the classical (and very old) finite element spaces specially tailored for the approximation of the Poisson equation in mixed form. In particular, properties \eqref{seqex2d} and \eqref{divboundinv} always hold.
\end{remark}

\subsection{Examples}\label{sec:exam0}
We now present three examples of  finite element spaces that can be used in the above framework.
 For each example, we specify the corresponding polynomial spaces used
on each element and describe the corresponding sets of degrees of freedom. We restrict our analysis to the case of triangles or tetrahedra; more general cases can also be considered when corresponding changes are made (see \cite{brezzi-fortin}).

Let us first fix the notation concerning the {\it spaces of polynomials}. For $m\geq 0$, we denote by $\Pp^{m}(\K)$ the space of polynomials defined on $\K$ of degree of at most $m$; the corresponding vector space is denoted by $\Pb^{m}(\K)=(\Pp^{m}(\K))^{2}$. A polynomial of degree $m\ge 3$ that vanishes throughout $\partial \K$ (hence it belongs to $H^1_0(\K)$) is called {\it a bubble (or an H-bubble) of degree $m$ over} $\K$. The space of bubbles of degree $m$ over $\K$ is denoted by $HB^m(\K)$.
and its vector-valued analogue by $\vect{HB}^m(\K)$.
We denote by
 $\Pp^{m}_{hom}(\K)$ the space of {\it homogeneous} polynomials of degree $m$, and we denote by $\bx^{\perp}$ the vector $(-x_2,x_1)$.

  For $m\ge 2$,
 \begin{equation}
\Pp_m^{+}(\K):=\Pp^m(\K) + HB^{m+1}(\K)\quad \Pb_m^{+}(\K):=\Pb^m(\K) + \vect{HB}^{m+1}(\K).
\end{equation}
 And, for  $m\ge 1$, we set
\begin{equation}
\BDM_{m}(\K):=\Pb^{m}(\K),\quad \RT_m(\K):=\Pb^{m}(\K)\oplus\bx\,\Pp^{m}_{hom}(\K).
\end{equation}
Moreover we set, for $d=2$ and $m\ge 0$
\begin{equation}
\TR_m(\K):=\Pb^{m}(\K)\oplus\bx^{\perp}\Pp^{m}_{hom}(\K).
\end{equation}
and for $d=3$ and $m\ge 0$ (see \cite{nedelec0})
\begin{equation}
\ND_m(\K):=\Pb^{m}(\K)\oplus\bx\wedge\Pb^{m}_{hom}(\K).
\end{equation}

We also consider some generalized bubbles: a vector-valued polynomial
of degree $m\ge 2$ that belongs to $\vect{H}_{0,n}(\div,\K)$ (hence
whose normal component vanishes throughout $\partial \K$) is called
{\it a D-bubble of degree $m$ over} $\K$. The space of D-bubbles of
degree $m$ over $\K$ is denoted by $\vect{DB}^m(\K)$.  Similarly a
vector valued polynomial of degree $m\ge d$ that belongs to
$\vect{H}_{0,t}(\curl,\K)$ (hence whose tangential components vanish
all over $\partial \K$) is called {\it a C-bubble of degree $m$ over}
$\K$. The space of C-bubbles of degree $m$ over $\K$ will be denoted
by $\vect{CB}^m(\K)$.

All the spaces used herein are well known and widely used. They are
usually referred to as {\it Brezzi-Douglas-Marini, Raviart-Thomas, and
  Rotated Raviart-Thomas} spaces, respectively.

The first example follows.

\medskip

\noindent {\it 1.  Raviart-Thomas } For $k\ge 1$, we take in each $\K$, $\Qk= \Pp^{k}(\K)$, and $\Rk:=\RT_{k}(\K)$.
The degrees of freedom in $\RT_{k}(\K)$ are
 \begin{equation}\label{dofs:rt}
\begin{aligned}
& \int_{e} \bu \cdot \n_e\, q\,ds \quad &\forall\, e\in\p \K,~\forall \, q \in \Pp^{k}(e), &&\\
&\int_{\K} \bu \cdot {\bf p} \, dx \quad &\forall\, {\bf p} \in \vect{\Pp}^{k-1}(\K). &&
 \end{aligned}
 \end{equation}
 As $\calQho$ is made of discontinuous piecewise polynomials, here and in the following examples the degrees of freedom in $\Qk$ can be taken in an almost arbitrary way.
The corresponding pair of spaces $(\Velh, \calQh)$ gives the classical Raviart-Thomas finite element approximation for second-order elliptic equations in mixed form, as introduced in \cite{raviart-thomas}.
It is well known and easy to check that the pair
$(\Velho, {\calQho})$  satisfies
\begin{equation}\label{divin}
\div(\Velho)= \calQho
\end{equation}
and that the property \eqref{divboundinv} is verified. We then take $\cMk:=\Pp^{k+1}(\K)$
and $\Mk:=\ND_{k}(\K)$ and note that
\begin{equation}\label{curlin}
\curl(\Nho)\subseteq \Velho \quad \qquad \curl(\Nbo)\subseteq \Velho
\end{equation}
and that the operator $\curl$ (for $d=2$ and $d=3$) has a continuous right inverse uniformly
bounded from $\Velho\cap\vect{H}_{0,n}(\div^0,\O)$ to
$\Nho$ and $\Nbo$ respectively; that is,
$$\exists \,C>0 \mbox{ such that } \forall h,~\forall \bv_h\in \Velho\cap\vect{H}_{0,n}(\div^0,\O)~~\exists \varphi\in \Nho, \mbox{ such that}$$

\vskip-0.8truecm

\begin{equation}\label{riub}
\curl\,\varphi=\bv_h\quad \mbox{and } \|\varphi\|_{1,\O}\le C\,\|\bv_h\|_{0,\O}.
\end{equation}

\noindent {\it 2. Brezzi-Douglas-Marini:} For $k \ge 1$, we take $\Qk= \Pp^{k-1}(\K)$, and $\Rk=\BDM_{k}(\K)$. The degrees of freedom for $\BDM_{k}(\K)$ are (see \cite{afw1}):
\begin{equation}\label{dofs:bdm}
\begin{aligned}
& \int_{e} \bu \cdot \n_e \,q \, ds \quad &\forall e\in\p \K,\;\forall\, q \in \Pp^{k}(e); &&\\
&\int_{\K} \bu \cdot \bv \, dx \quad &\forall\, \bv \in \TR_{k-2}(\K) \quad k\ge 2 \mbox{ and } d=2, &&\\
&\int_{\K} \bu \cdot \bv \, dx \quad &\forall\, \bv \in \ND_{k-2}(\K)\quad k\ge 2 \mbox{ and } d=3. &&
 \end{aligned}
 \end{equation}
The resulting finite element pair  $(\Velh,\calQh)$ is also commonly used for the approximation of second-order elliptic equations in mixed form introduced in \cite{brezzi-douglas-marini} for $d=2$ and  in \cite{nedelec2, brezzi-douglas-duran-fortin} for $d=3$. Also in this case it has been established that  the pair
$(\Velho, \calQho)$  verifies the  properties of \eqref{divin}
and \eqref{divboundinv}. We then take $\cMk:=\Pp^{k+1}(\K)$, and $\Mk:=\ND_{k+1}(\K)$ and note that \eqref{curlin} and \eqref{riub} are also satisfied.



\noindent {\it 3. Brezzi-Douglas-Fortin-Marini:} For $k\ge 1$, we take $\Qk= \Pp^{k}(\K)$ and $\Rk=\BDFM_{k+1}(\K)$, which can be written as $\BDFM_{k+1}=\BDM_{k}(\K)+\vect{DB}_{k+1}(\K)$. The degrees of freedom for $\BDFM_{k+1}(\K)$, though similar to the previous ones, are given here:
\begin{equation}\label{dofs:bdfm}
\begin{aligned}
& \int_{e} \bu \cdot \n_e \,q \, ds \quad &\forall e\in\p \K,\;\forall\, q \in \Pp^{k}(e); &&\\
&\int_{\K} \bu \cdot \bv \, dx \quad &\forall\, \bv \in \TR_{k-1}(\K)  \qquad d=2, &&\\
&\int_{\K} \bu \cdot \bv \, dx \quad &\forall\, \bv \in \ND_{k-1}(\K)\qquad  d=3. &&
 \end{aligned}
 \end{equation}
 The resulting finite element pair $(\Velh,\calQh)$ gives the
 triangular analogue of the element BDFM$_{k}$ introduced in
 \cite{brezzi-douglas-fortin-marini} for the approximation of
 second-order elliptic equations in mixed form. It is easy to check
 that the pair $(\Velho, \calQho)$ verifies \eqref{divin} and
 \eqref{divboundinv}. We then take $\cMk:=\Pp_{k+1}^{+}(\K)$ and
 $\Mk:=\ND_{k}(\K)+\vect{CB}_{k+1}(\K)\cap\ND_{k+1}(\K)$ and note that
 \eqref{curlin} and \eqref{riub} hold.

 The three choices above are quite similar to each other, and the best
 choice among them generally depends on the problem and the way in
 which the discrete solution is to be used.  We also use basic
 approximation properties: for instance, we recall that a constant $C$
 exists such that for all $\K\in\Th$ and for all $\bv$, e.s. in
 $\vect{H}^s(\K)$, an interpolant $\bv^{I} \in \Rk$ exists such that
\begin{equation}\label{aprox:00}
\|\vect{v}-\vect{v}^{I}\|_{0,\K} +\hK|\vect{v}^{I}|_{1,\K}   \leq C \hK^s |\vect{v}|_{s,\K},\quad s\le k+1.
\end{equation}

\section{The discontinuous Galerkin $H(\div;\O)$-conforming method}\label{sec:2}
To introduce our DG-approximation, we start by defining, for any $\bu
,\bv \in \vect{H}^{2}(\Th)$ and any $p , q\in L^2(\O)/\re$, the
bilinear forms
\begin{equation}\label{stokes0}
\begin{aligned}
{\it A}_h(\bu,\bv)&=2\nu\left[ (\Gsym{u}: \Gsym{v})_{\Th} -\langle \av{\Gsym{u}}: \jump{\bv}\rangle_{\Eho}  -\langle \jump{\bu}: \av{\Gsym{v}} \rangle_{\Eho}\right]&&\\
& \,\,-2\nu\left[ \langle \Gsym{u}\n , (\bv\cdot\n)\n \rangle_{\Ehb} +\langle (\bu\cdot\n)\n , \Gsym{v}\n  \rangle_{\Ehb}\right] &&\\
&+2\nu\left[\sum_{e\in\Eho}\alpha h_e^{-1} \int_e\jump{\bu}: \jump{\bv}\ds +\sum_{e\in\Ehb} \alpha h_e^{-1}\int_e (\bu\cdot \n)(\bv\cdot \n) \ds\right] &&\\
{\it B}_h(\bv,q)&=-(q,\div\,\bv)_{\Th}\qquad \forall\, \bv \in \vect{H}^{2}(\Th), \forall\, q\in  L^2(\O){/\re} \;&&
\end{aligned}
\end{equation}
where as usual $\alpha$ is the penalty parameter that we assume to be positive and large enough.

It is easy to check that the solution $(\bu,p)$ of \eqref{var:0} verifies:
\begin{equation}\label{DG-gen}
\left\{
\begin{aligned}
{\it A}_h(\bu,\bv)+{\it B}_h(\bv, p)&=(\f,\bv) \qquad &\forall\, \bv \in \vect{H}^{2}(\Th) &&\\
{\it B}_h(\bu, q)&=0\qquad &\forall\, q \in L^2(\O)/\re. &&
\end{aligned}\right.
\end{equation}
For a general DG approximation, we now replace the spaces $ \vect{H}^{2}(\Th)$ and $L^2(\O){/\re}$ with the discrete ones $ \calX$ and $ \calQh$, respectively.
Following \cite{dominik0}, we choose for $ (\calX, \calQh)$ one of the pairs $(\Velho,\calQh)$  of the previous  examples in order to get a global divergence-free approximation.

More generally, we can choose a pair $(\Velho,\calQh)$ in order to find a third space $\Nho$ in such a way that \eqref{seqex2d}, \eqref{divin}, \eqref{divboundinv}, \eqref{curlin}, and \eqref{riub} are satisfied. This set of assumptions will come out several times in the sequel and, therefore, it is helpful to give it
a special name.
\begin{definition}\label{defH0}
In the above setting, we say that the three spaces $(\Velho,\calQh,\Nho)$ (resp. $(\Velho,\calQh,\Nb)$) satisfy Assumption {\rm{\bf H0}} if
\eqref{seqex2d} (resp. \eqref{seqex3d}), \eqref{divin}, \eqref{divboundinv}, \eqref{curlin} and \eqref{riub} are satisfied.
\end{definition}


We note that, according to the definition of $\Velho$,  the normal component of  any $\bv \in \Velh$ is continuous on the internal edges and vanishes on the boundary edges. Therefore, by splitting a vector $\bv\in \Velh$ into its tangential and normal components $\bv_n$ and $\bv_t$
\begin{equation}\label{split:0}
\bv_n:= (\bv\cdot\n)\n, \quad \bv_t:=(\bv\cdot{\bf t}){\bf t}\equiv \bv-\bv_n ,
\end{equation}
we have
\begin{equation}\label{contn:0}
\forall\, e\in \Eh\quad \int_{e}\jump{\bv_n}: \vect{\tau}\ds =0 \quad \forall\, \vect{\tau}\in \vect{\mathcal H}^1{(\Th)},
\end{equation}
implying that
\begin{equation}\label{contn:0t}
\forall\, e\in \Eh\quad \int_{e}\jump{\bv}: \vect{\tau}\ds =
 \int_{e}\jump{\bv_t}: \vect{\tau}\ds\quad \forall\, \vect{\tau}\in \vect{\mathcal H}^1{(\Th)}.
\end{equation}

The resulting approximation to  \eqref{var:0}, therefore, becomes:
{\it Find  $(\bu_h, p_h)$ in $\Velho  \times {\calQho}$ such that}
\begin{equation}\label{dg:bdm0}
\left\{
\begin{aligned}
a_h(\bu_h,\bv)+b(\bv, p_h)&=(\f,\bv) \qquad &\forall\, \bv \in \Velho&&\\
b(\bu_h, q)&=0\qquad &\forall\, q \in{\calQho}, &&
\end{aligned}
\right .
\end{equation}
where
\begin{equation}\label{method:0}
\begin{aligned}
a_h(\bu,\bv)&:=2\nu \left[ (\Gsym{u}: \Gsym{v})_{\Th} -\langle \av{\Gsym{u}}: \jump{\bv_t}\rangle_{\Eho}
-\langle \jump{\bu_t}: \av{\Gsym{v}} \rangle_{\Eho}\right]&&\\
&+2\nu\alpha\sum_{e\in\Eho}  h_e^{-1} \int_e \jump{\bu_t}: \jump{\bv_t}\,\ds \qquad \forall\, \vect{u}, \vect{v} \in  \Velho,&&\\
b(\bv,q)&:=-(q,\div\,\bv)_{\O}\qquad \forall\,\bv \in \Velho,~\forall q\in \calQho.
\end{aligned}
\end{equation}
{\bf Consistency} The consistency of the formulation \eqref{dg:bdm0} can be checked by means of the usual DG-machinery. In this case, it is sufficient to compare \eqref{stokes0} and \eqref{method:0} and to observe that if $(\bu,p)$ is the solution of \eqref{var:0}, then
\begin{equation*}
A_h(\bu,\bv_h)\equiv a_h(\bu,\bv_h), \quad B_h(\bv_h,p)\equiv b(\bv_h,p), \quad\forall \bv_h \in \Velho\subseteq \vect{H}_{0,n}(\div;\O),
\end{equation*}
Further, it is evident that, $B_h(\bu,q_h)\equiv b(\bu,q_h)$ for all $q_h\in\calQho$. Hence, as $(\bu,p)$ verifies \eqref{DG-gen}, it also verifies \eqref{dg:bdm0}; that is,
\begin{equation}\label{consistency}
\left\{
\begin{aligned}
a_h(\bu,\bv)+b(\bv, p)&=(\f,\bv) \qquad &\forall\, \bv \in \Velho&&\\
b(\bu, q)&=0\qquad &\forall\, q \in{\calQho}. &&
\end{aligned}
\right .
\end{equation}
Thus, consistency is proved.

To prove the existence and uniqueness of the solution of \eqref{dg:bdm0} and to obtain the optimal error bounds, we need to define suitable norms.
We define the following semi-norms
\begin{equation*}
 |\vect{v}|^{2}_{1,h}=\sum_{\K\in \Th} \|\nabla \bv\|_{0,\K}^{2}, \quad |\jump{\vect{v}}|^{2}_{\ast}:=\sum_{e\in \Eho} h_e^{-1} \| \jump{\vect{v}}\|_{0,e}^{2}, \qquad \forall\, \vect{v} \in \vect{H}^{1}(\Th),
\end{equation*}
and norms
\begin{equation}\label{norme}
\begin{aligned}
\|\bv\|_{DG}^{2}:&=2\nu\, |\bv|_{1,h}^{2} +2\nu \,|\jump{\bv_t}|_{\ast}^{2}\quad & \bv \in \vect{H}^{1}(\Th),&&\\
\trin\bv\trin^{2}:&=\|\bv\|^{2}_{DG}+\sum_{\K\in\Th}2\nu\,h_{\K}^{2}|\Gsym{v}|_{1,\K}^{2}  \quad &\bv \in \vect{H}^{2}(\Th). &&
\end{aligned}
\end{equation}
We also remark that the seminorms defined in \eqref{norme} are actually norms with the additional
requirement  that $\bv\in\vect{H}_{0,n}(\div;\O)$. We also observe that when restricted to discrete functions $\bv \in \Velho$, the $\|\cdot\|_{DG}$-norm and the $\trin\cdot\trin$ are equivalent (using inverse inequality).
Continuity can easily be shown for both bilinear forms:
 \begin{equation*}
\begin{aligned}
|a_h(\bu, \bv)|&\leq  \trin\bu\trin \, \trin\bv\trin \qquad &\forall\, \bu, \, \bv \in \vect{H}^{2}(\Th), &&\\
|b(\bv,q)| &\leq  \|\bv\|_{1,h}\|q\|_{0,\O}  \qquad &\forall\, \bv \in \vect{H}^{1}(\Th),  \,\,\, q\,\in L^2(\O)/\mathbb{R}\;.&&
\end{aligned}
\end{equation*}
Following \cite{brezzi-fortin}, the existence and uniqueness of the approximate solution and optimal error bounds are guaranteed if the following two conditions are satisfied:
\begin{description}
\item[{\bf (H1)}] {\it coercivity}: $\exists\,\, \gamma>0$ independent of the mesh size $h$ such that
\begin{equation}\label{a1}
a_h(\bv,\bv)\geq \gamma \|\bv\|_{DG}^{2} \qquad \forall \, \bv\in \Velho .
\end{equation}
\item[{\bf (H2)}] {\it inf-sup condition}: $\exists\,\, \beta>0$ independent of the mesh size $h$ such that
\begin{equation}\label{a2}
\sup_{\bv \in \Velho} \frac{ (\div \,\bv, q_h)_{\O}}{\|\bv \|_{DG}}\geq \beta \|q_h\|_{0,\O} \quad \forall\, q_h \in {\calQho}.
\end{equation}
\end{description}
Condition {\bf (H2)} is a consequence of the {\it inf-sup} condition that holds for the continuous problem \eqref{var:0}:
\begin{equation*}
\exists\, \beta>0 ~\mbox{s.t. }~\forall h,~\forall q_h \in \calQho\quad \exists \,\bv \in
\vect{H}^{1}(\O)
:~\div \,\bv= q_h \mbox{ and }\|\bv\|_{1,\O}\le \frac{1}{\beta}\|q_h\|_{0,\O}.
\end{equation*}
It is well known that for all the families considered here an
interpolation operator $\bv\rightarrow\bv^I \in \Velho$ exists that
verifies \eqref{aprox:00} (in particular for $s=1$), and
\begin{equation*}
\div \,\bv^I=\div \,\bv ~(=q_h).
\end{equation*}
By observing that $\jump{\bv}=0$ on the internal edges as $\bv\in\vect{H}^{1}(\O)$,
and by using the Agmon trace inequality \cite{agmon}
and \eqref{aprox:00} (for $s=1$), we have
\begin{equation}\label{eq:jump-semi-norm}
|\jump{\bv^I}|^2_*:=\sum_{e\in \Eho} h_e^{-1} \| \jump{\vect{v}^I_t}\|_{0,e}^{2}=\sum_{e\in \Eho} h_e^{-1} \| \jump{(\bv^I-\bv)_t}\|_{0,e}^{2}\le C\,|\bv|^2_{1,\O}.
\end{equation}
Hence, again using \eqref{aprox:00}, we deduce that
\begin{equation*}
\|\bv^I\|_{DG}\le C\,|\bv|_{1,\O}.
\end{equation*}
Thus \eqref{a2} is proved.\\

In order to prove \eqref{a1} we need to extend \eqref{kornn0} from Lemma \ref{korn:cont:new} to spaces of discontinuous
vectors. We have therefore the following result. Also see Appendix \ref{app:A} for further comments on the validity of the result in three dimensions.

\begin{lemma}\label{korn}
Let $\Velho$ be a piecewise polynomial subspace of $\vect{H}_{0,n}(\div;\O)$.  Then, $\exists\,C_K>0$ independent of $h$ such that
\begin{equation}\label{korn1}
|\bv|_{1,h}^{2} \leq C_{K} \left( \| \Gsym{v} \|_{0,\Th}^{2} + \sum_{e\in \Eho} h_e^{-1}\| \jump{ \bv_t}\|_{0,e}^{2} \right), \quad \forall\, \bv \in \Velho.
\end{equation}

\end{lemma}
\begin{proof}
To show \eqref{korn1}, a direct application of \cite[Inequality
(1.14)]{Brenner04}
to $\bv \in  \Velho$ 
gives
\begin{equation}\label{daSue}
|\bv|_{1,h}^{2} \leq C_{K} \left( \| \Gsym{v} \|_{0,\Th}^{2} + \sum_{e\in \Eho} h_e^{-1}\| \jump{ \bv_t}\|_{0,e}^{2} +\sup_{  \substack{
\vect{\eta}\in \vect{L}^2(\O) \\
\|\vect{\eta}\|_{0,\O}=1, \, \int_{\O} \!\vect{\eta}=0}}  \left(\int_{\O} \vect{v}\cdot\vect{\eta} dx\right)^{2} \right)\;,
\end{equation}

We now show that the last term in \eqref{daSue} can be bounded by the first two. We claim that
\begin{equation}\label{aprov:00}
\sup_{  \substack{
\vect{\eta}\in \vect{L}^2(\O) \\
\|\vect{\eta}\|_{0,\O}=1, \, \int_{\O} \!\!\vect{\eta}=0}}  \left(\int_{\O} \vect{v}\cdot\vect{\eta} dx\right)^{2} \leq C\Big( \| \Gsym{v} \|_{0,\Th}^{2} + \sum_{e\in \Eho} h_e^{-1}\| \jump{ \bv_t}\|_{0,e}^{2} \Big).
\end{equation}
There are surely many ways of checking \eqref{aprov:00}. Here, we propose one. For
$\bv\in\Velho$ and $\vect{\eta}\in \vect{L}^2(\O)$ with $\int_{\O}\vect{\eta}\dx=0$, we set
\begin{equation*}
\mathcal{I}(\vect{v},\vect{\eta}):= \int_{\O} \vect{v}\cdot\vect{\eta} \dx,
\end{equation*}
and we want to prove that
\begin{equation}\label{boundtoprove}
\mathcal{I}(\vect{v},\vect{\eta})\leq C\Big( \| \Gsym{v} \|_{0,\Th}^{2} + \sum_{e\in \Eho} h_e^{-1}\| \jump{ v_t}\|_{0,e}^{2} \Big)^{1/2}\|\vect{\eta}\|_{0,\O}
\end{equation}
that will easily give \eqref{aprov:00} taking the supremum with respect to $\vect{\eta}$ with
$\|\vect{\eta}\|_{0,\O}=1$. To prove \eqref{boundtoprove} for every $\vect{\eta} \in \vect{L}^2(\O)$ with $\int_{\Omega}\vect{\eta}\dx=0$, we consider the following auxiliary elasticity problem: {\it Find $\vect{\chi}\in \vect{H}^1_{0,n}$ such that:}
\begin{equation}\label{el:aux}
(\Gsym{\chi},\Gsym{v})_{0,\O}=(\vect{\eta},\bv)_{0,\O}\qquad\forall \bv\in \vect{H}^1_{0,n}.
\end{equation}
Thanks to \eqref{kornn0} problem \eqref{el:aux} has a unique solution, and we set
\begin{equation}\label{el:aux2}
\vect{\tau}:=\Gsym{\chi}.
\end{equation}
We note that
as natural boundary condition for \eqref{el:aux} we easily have
\begin{equation}\label{bc:aux}
(\vect{\tau})_{nt}\equiv (\Gsym{\chi}\cdot\vect{n})\cdot\vect{t}=0 \qquad \mbox{on }\Gamma,
\end{equation}
where $\vect{t}$ is any tangent unit vector to $\Gamma$.

Due to well-known results on the regularity of the solutions of PDE systems on
polygons, the solution $\taub$ of
\eqref{el:aux}-\eqref{el:aux2} (which, a priori, on a totally general domain would only be in $(L^2(\O))^{2\times 2}_{sym}$) satisfies  the following a priori estimate: {\it there exists a $p>2$  (depending on the geometry
of $\O$) and a constant $C_p$  such that for all $\vect{\eta}\in\vect{L}^2(\O)$ the corresponding $\taub$ satisfies}
\begin{equation}\label{apri:korn}
\|\taub\|_{(L^p(\O))^{2\times 2}_{sym}}+\|{\rm\bf div}\vect{\tau}\|_{0,\Omega} \leq C_p\|\vect{\eta}\|_{0,\O}.
\end{equation}




The proof of the following proposition (actually, in two or three dimensions) is given in Appendix \ref{app:A}.


\begin{proposition}\label{perappA}
Let $\K$ be a triangle with minimum angle $\theta>0$, and let $e$ be an edge of $\K$. Then for every $p>2$
and for every integer $k_{max}$,
a constant $C_{p,\theta,kmax}$ exists such that
\begin{equation}\label{propA}
\int_e\vect{v}\cdot(\taub\cdot\n)\ds\le C_{p,\theta,kmax}\,
h_{\K}^{-1/2}\|\vect{v}\|_{0,e}\;
(h_{\K}\|{\rm\bf div}\vect{\tau}\|_{0,\K}+ h_{\K}^{\frac{p-2}{p}}\|\taub\|_{0,p,\K})
\end{equation}
for every $\taub\in(L^p(\O))^{2\times 2}_{sym}$ with divergence in $\vect{ L}^2(\K)$ and for every $\vect{v} \in \Pb^{k_{max}} (e)$.
\end{proposition}
Then we have
\begin{equation}\label{byparts}
\begin{aligned}
\mathcal{I}(\vect{v},\vect{\eta})&= \int_{\O} \vect{v}\cdot\vect{\eta} \dx=\,-\int_{\O} \vect{v}\cdot({\rm\bf div}\vect{\tau}) \dx  \\
&=( \Gsym{v}: \vect{\tau})_{\Th}-  \langle \jump{\bv_t} : \av{ \vect{\tau}} \rangle_{\Eho} 
\end{aligned}
\end{equation}
having taken into account that at the interelement boundaries the normal component of $\bv$ is continuous and on $\Gamma$ both the normal component of $\bv$ and $(\vect{\tau})_{nt}$ are zero.

At this point, we can apply \eqref{propA} to each $e$ of the last term in \eqref{byparts}. We apply the usual Cauchy-Schwarz inequality on the first term  and we use instead
the generalized H\"older inequality (with $q=1/2$ and $r=2p/(p-2)$, so that $\frac{1}{p}+\frac{1}{q}+\frac{1}{r}=1$) on the second one. Then we obtain
\begin{align}
 \sum_{e\in \Eho}\int_e \jump{\bv_t}& : \av{ \vect{\tau}}\ds \le \sum_{\K\in\Th}\sum_{e\in\partial\K} C_{p,\theta,kmax}\, \Big(h_{\K}^{-1/2}\|\vect{v}\|_{0,e}\,h_{\K}\|{\rm\bf div}\vect{\tau}\|_{0,\K}+ h_{\K}^{-1/2}\|\vect{v}\|_{0,e}\|\,h_{\K}^{\frac{p-2}{p}}\|\taub\|_{0,p,\K}\Big)
 \nonumber\\
 &\le C |\jump{\vect{v}_t}|_{\ast}\,h\,\|{\rm\bf div}\taub\|_{0,\O}
 +
 C\Big(\sum_{e\in\Eho}h_e^{-1}|\jump{\vect{v}_t}|^2_{0,e}\Big)^{1/2}
 \Big(\sum_{e\in\Eho}\|\taub\|_{0,p,\K(e)}^p\Big)^{1/p}
 \Big(\sum_{e\in\Eho}h_e^{\frac{p-2}{p}r}\Big)^{1/r}\label{usapropA}\\
 &\le C h\, |\jump{\vect{v}_t}|_{\ast}\,\|{\rm\bf div}\taub\|_{0,\O}+C\,|\jump{\vect{v}_t}|_{\ast}\,\|\taub\|_{0,p,\O}\,\mu({\O})^{1/r}\nonumber
\end{align}
 where for each $e\in \Eho$ with $e=\partial \K^{+}\cap \partial\K^{-}$, the set $\K(e)$ refers to $\K(e):=\K^{+}\cup\K^{-}$. In the second line,  $\mu(\O)$ denotes the measure of the domain $\O$, whereas the constant $C$ still depends on
$p$, $k_{max}$ and on the maximum angle in the decomposition $\Th$.

From \eqref{byparts}, \eqref{usapropA}, and the bound \eqref{apri:korn} we then obtain
\begin{equation}
\left|\mathcal{I}(\vect{v},\vect{\eta})\right| \leq
 C \left(\| \Gsym{v}\|_{0,\Th}+  |\jump{\vect{v_t}}|_{\ast} \right) \|\vect{\eta}\|_{0,\O}
\end{equation}
which gives \eqref{boundtoprove}. Thus the proof of the lemma is complete.
\end{proof}

\begin{remark}
The fact that in inequality \eqref{korn1} only the jumps over the interior edges $e\in \Eho$ (but not on the boundary edges) are included, prevents a direct and straightforward application of the results  from \cite{brenner03}. The proof presented here is surely too elaborate, and we believe that a simpler proof is possible. However some of the machinery used here is likely to be of use elsewhere. Therefore,
we decided that it would be worthwhile to present the proof we have obtained to date.
\end{remark}

The stability of $a_h(\cdot,\cdot)$ in the $\|\cdot\|_{DG}$-norm can now be easily checked  with the
usual DG machinery. We have
\begin{equation*}
\left| \int_{e} \av{\Gsym{v}}: \jump{\bv_t} \ds\right| \leq h^{1/2}\|\av{\Gsym{v}}\|_{0,e} \|h^{-1/2}\jump{\bv_t}\|_{0,e},
\end{equation*}
which when we proceed as in \cite{abcm} (or as in \eqref{usapropA} with $p=2$) yields
\begin{equation}\label{cross2}
\left| \sum_{e\in\Eho} \int_{e} \av{\Gsym{v}}: \jump{\bv_t} \ds \right|  \leq C|\bv|_{1,h}\,|\jump{\bv_t}|_{\ast}.
\end{equation}
Using \eqref{cross2} in \eqref{method:0}, we then have
\begin{equation*}
a_h(\bv, \bv) \geq  2\nu\|\Gsym{v}\|^{2}_{0,\Th} +2\nu \,\alpha \|\jump{\bv_t}|_{\ast}^2-4\nu C|\bv|_{1,\Th}\,|\jump{\bv_t}|_{\ast} .
\end{equation*}
Now using the Korn inequality \eqref{korn1} and the usual arithmetic-geometric mean inequality, we easily have
a big enough $\alpha$ :
\begin{equation*}
a_h(\bv, \bv) \geq \gamma \|\bv \|_{DG}^{2} \quad \forall\, \bv \in \Velho .
\end{equation*}

We close this section with the following theorem.
\begin{theorem}\label{teo0}
Let $(\Velho,\calQho)$ be as in one of our three examples. Then
problem \eqref{dg:bdm0} has a unique solution $(\bu_h,p_h)\in \Velho\times\calQho$ that verifies
\begin{equation}\label{divzero}
\div \,\bu_h=0 \quad \mbox{in }\O.
\end{equation}
 Moreover, there exists a positive constant $C$, independent of $h$, such that for every $\bv_h\in \Velho$ with $\div\,\bv_h=0$ and for every $q_h\in\calQho$ the following estimate holds:
\begin{equation}\label{error-estimate}
\|\bu-\bu_h\|_{DG}\le C\, \|\bu-\bv_h\|_{DG},\qquad \|p-p_h\|_{0,\O}\le C\,(\|p-q_h\|_{0,\O}+\|\bu-\bv_h\|_{DG}),
\end{equation}
with $(\bu,p)$ solution of \eqref{var:0}.
\end{theorem}
\begin{proof}
The existence and uniqueness of the solution of \eqref{dg:bdm0} follow from \eqref{a1}-\eqref{a2}. The divergence-free property \eqref{divzero} is implied by \eqref{divin}, which holds for all our choices of spaces. Let $\bv_h\in \Velho$ also be divergence-free; then we obviously have that $b(\bv_h-\bu_h,q)=0$ for every $q\in L^2(\O)/\re$. In particular, $b(\bv_h-\bu_h,p-p_h)=0$. Hence,
from the coercivity \eqref{a1}, consistency \eqref{consistency}, and continuity of $a_h(\cdot,\cdot)$ we deduce immediately
\begin{equation*}
\gamma\|\bv_h-\bu_h\|^2_{DG}\le a_h(\bv_h-\bu_h,\bv_h-\bu_h)=a_h(\bv_h-\bu,\bv_h-\bu_h)\le \|\bv_h-\bu\|_{DG}\|\bv_h-\bu_h\|_{DG}.
\end{equation*}
On the same basis we deduce that the first estimate in \eqref{error-estimate} follows by triangle inequality.
For every $\bw_h\in\Velho$, using the consistency and continuity of $a_h(\cdot,\cdot)$, we have
\begin{align}\label{perstimap}
b(\bw_h, q_h-p_h)&=
b(\bw_h, q_h-p)+b(\bw_h, p-p_h)
= b(\bw_h, q_h-p)-a_h(\bu-\bu_h,\bw_h)\nonumber\\
&\le (\|q_h-p\|_{0,\O}+\|\bu-\bu_h\|_{DG})\|\bw_h\|_{DG}.
\end{align}
By dividing \eqref{perstimap} by $\|\bw_h\|_{DG}$ and then using the {\it inf-sup} condition \eqref{a2},  we immediately deduce that
$$
\beta \|q_h-p_h\|_{0,\O}\le \|q_h-p\|_{0,\O}+\|\bu-\bu_h\|_{DG},
$$
and that the second estimate in \eqref{error-estimate} follows again by triangle inequality.
\end{proof}

\begin{remark}\label{rt0}
In the assumptions of Theorem\ref{teo0}, we could obviously consider any trio
of finite element spaces satisfying {\bf H0}. However, for choices like ${\bf RT_0}$,
not considered in our three examples, the estimate \eqref{error-estimate} could be
meaningless, as the term $\|\bu-\bvh\|_{DG}$ does not, in general, go to zero
with $h$. Still, this choice could be profitably used, in some cases, as a
preconditioner, as it does satisfy {\bf H0}, {\bf H1}, and {\bf H2}.
\end{remark}


\section{Discrete Helmholtz decompositions}\label{ottob0}

In this section we provide results related to the discrete Helmholtz
decomposition, introduced in Section \ref{sec:1} that plays a key
role in the design of the preconditioner. We wish to note that
Discrete Helmholtz or Hodge decompositions have been shown and used in
several contexts for similar spaces but with other boundary conditions (typically, homogeneous Dirichlet) in \cite{bf86, brezzi-fortin-stenberg, afw00,afw02}. A nice and short
proof in the language of Finite Element Exterior Calculus can be also
found in \cite[p. 72]{afw1}. Here, together with the proof of the
decomposition  with our boundary conditions, we provide an estimate in the DG-norm for the components
in the splitting, that will be essential in the analysis of the
solver, and that, to the best of our knowledge, has not been obtained or used
in any previous work.

\bigskip

So far, we have assumed that the computational domain $\O$ is a polygon (or polyhedron). From now on, for the sake of simplicity,
we are going to work under the stronger assumption that $\O$ is {\it a convex} polygon or polyhedron. As is well known, this
allows the use of better regularity  results, and in particular the $H^2$-regularity for elliptic second-order operators.

Following \cite{brezzi-fortin} we define the discrete gradient operator $\gradh : {\calQho} \lor \Velho$ as
\begin{equation}\label{def-gradh}
(\gradh q_h, \bvh)_{0,\O}:=-(q_h,\div\,\bvh)_{0,\O} \quad \forall\, \bvh \in \Velho .
\end{equation}
\begin{lemma}\label{helmholtz}
  Assume that together the three spaces $(\Velho, \calQho, \Nho)$ (resp. $(\Velho, \calQho, \Nb)$)  satisfy
  assumption $\bf H0$ (given in Definition \ref{defH0}). Then, in $d=2$, for any
  $\bvh\in \Velho$ a unique $q_h\in {\calQho}$ and a
  unique $\varphi_h\in \Nho$ exist such that
\begin{equation}\label{decomp-helm2}
\bvh=\gradh q_h + \curl \,\varphi_h,
\end{equation}
that is,
\begin{equation*}
 \Velho = \gradh ({\calQho}) \oplus \curl \,\Nho .
\end{equation*}
If $d=3$, there exists a $\vect{\psi} \in \Nb$ such that
\begin{equation}\label{decomp-helm3}
\bvh=\gradh q_h + \curl \,\vect{\psi}_h,
\end{equation}
and therefore
\begin{equation*}
 \Velho = \gradh ({\calQho}) \oplus \curl \,\Nb .
\end{equation*}
Moreover, in both cases there exists a constant $C$ independent of $h$ such that the following estimate holds:
\begin{equation}\label{est:s-v0}
\|\gradh q_h\|_{DG}\leq C\|{\rm {div}} \,\bvh\|_{0,\O} .
\end{equation}
\end{lemma}
We present the proof in two dimensions; see however  Remark \ref{por3D} after this proof, where the differences for the case $d=3$ are discussed.
\begin{proof}
  For $\bv_h \in \Velho$, consider the auxiliary problem:
 \begin{equation}\label{poisson0}
 -\Delta q=\div\,\bvh \quad \mbox{in  } \O, \qquad \frac{\partial q}{\partial n}=0 \quad \mbox{on  } \partial \O,\quad \mbox{ and }\int_{\O}q\dx=0.
 \end{equation}
Owing to the boundary conditions in $\Velho$, we have that $\div\,\bvh$ has zero mean value in $\O$. Hence, problem \eqref{poisson0} has a unique solution, that satisfies
 \begin{equation}\label{aux2:01}
\|q\|_{2,\O} \leq C_{reg} \|\div \,\bvh\|_{0,\O}.
\end{equation}
We write \eqref{poisson0} in mixed form:
\begin{equation*}
\sig=-\nabla q  \mbox { in  } \O, \quad \div \,\sig = \div \,\bvh  \mbox{ in   } \O, \quad \sig \cdot\n = 0\mbox{  on  } \partial\O.
\end{equation*}
and  we consider directly the approximation of the mixed formulation: {\it Find $(\sig_h , q_h) \in  \Velho \times {\calQho}$ such that :}
 \begin{equation}\label{poisson:1}
 \left\{ \begin{aligned}
 &&(\sig_h, \taub)_{0,\O} - (q_h, \div \,\taub)_{0,\O} &=0 \qquad &\forall\, \taub \in  \Velho , &&\\
 &&(\div \,\sig_h ,s_h)_{0.\O} &=(\div \,\bvh, s_h)_{0,\O} \quad &\forall\, s_h\in {\calQho}. &&
 \end{aligned}\right.
 \end{equation}
 Problem \eqref{poisson:1} obviously has a unique solution, which moreover satisfies
 \begin{equation}\label{approx:01}
 \|\sig-\sig_h\|_{0,\O}\le C \, h\,|\sig |_{1,\O}\le\,C C_{reg}\,h\,\|\div\, \bv_h\|_{0,\O},
 \end{equation}
given that \eqref{aux2:01} was used in the last step. As both $\bvh$ and $\sig_h$ are in $\Velho$ (and
as \eqref{divin} holds), the second
 equation in \eqref{poisson:1}  directly implies that
 \begin{equation*}
 \div\,(\sig_h-\bvh) =0.
 \end{equation*}
 Hence, the exact sequence \eqref{seqex2d} implies that
 \begin{equation}\label{pot2d}
 \mbox{ a unique } \,\varphi_h\in \Nho  \quad \mbox{exists such that}\quad \sig_h-\bvh = \curl \,\varphi_h.
 \end{equation}
%
Next, by using the first equation in  \eqref{poisson:1} and then applying definition \eqref{def-gradh}, we deduce that
 \begin{equation*}
 (\sig_h,\taub)_{0,\O}=(q_h,\div\,\taub)_{0,\O}= -(\gradh q_h, \taub)_{0,\O}  \quad \forall\, \taub \in\Velho,
\end{equation*}
which implies $\sig_h=-\gradh q_h$, that joined to \eqref{pot2d} gives \eqref{decomp-helm2}.

In order to prove \eqref{est:s-v0}, we recall that
\begin{equation}\label{aux2:07}
 \|\gradh q_h\|^{2}_{DG}=\|\sig_h\|^{2}_{DG}= \|\nabla \sig_h\|_{0,\Th}^{2} + |\jump{(\sig_h)_t}|_{\ast}^{2} .
 \end{equation}
 For the first term, by adding and subtracting the interpolant $\sig^I$ of $\sig$ and then using inverse inequality and \eqref{aprox:00}, we have:
 \begin{alignat}{1}
 \|\nabla \sig_h\|_{0,\Th} &\leq \|\nabla (\sig_h-\sig^I)\|_{0,\Th}+\|\nabla \sig^I\|_{0,\Th}\nonumber\\
 &\le C_{inv}h^{-1}\| \sig_h-\sig^I\|_{0,\Th}+ C \|\nabla \sig\|_{0,\Th} .
 \end{alignat}
  From triangle inequality, \eqref{approx:01}, and standard approximation properties  (see \eqref{aprox:00}), we have
 \begin{equation}\label{aux2:09}
 \|\nabla \sig_h\|_{0,\Th}\le C\, \|\div \,\bvh\|_{0,\O}.
 \end{equation}
The jump term in \eqref{aux2:07} is estimated similarly. First, we remark that  $\vect{\sig}=-\nabla q$ with $q\in H^{2}(\O)$ so that $\jump{\vect{\sig}}=0$ on each $e\in \Eho$, and therefore
\begin{equation*}
|\jump{(\vect{\sig}_h)_t} |_{\ast}^{2}= |\jump{(\vect{\sig}_h)_t-\vect{\sig}_t} |_{\ast}^{2}.
\end{equation*}
Then, using  Agmon trace inequalities  \eqref{approx:01}  and the boundedness of $\sig_h$ and $\sig$, we have
  \begin{align*}
|\jump{(\vect{\sig}_h)_t-\vect{\sig}_t} |_{\ast}^{2}=&\sum_{e\in\Eho} h_e^{-1}\|\jump{(\vect{\sig}_h)_t-\vect{\sig}_t}\|_{0,e}^{2} &&\\
 & \leq
C_{t} h^{-2}\|\sig_h-\sig \|_{0,\Th}^{2}+C_t\|\nabla (\sig_h-\sig) \|_{0,\Th}^{2} &&\\
&\leq CC_{reg} \|\div \,\bv_h\|^2_{0,\O} .&&
\end{align*}
Thus the proof is complete.
\end{proof}

\begin{remark}\label{por3D}
For $d=3$, instead of \eqref{pot2d},  the exact sequence \eqref{seqex3d} property implies
  \begin{equation*}
 \exists \,\vect{\psi}_h \in \Nb   \quad \mbox{such that}\quad \sig_h-\bvh = \curl\, \vect{\psi}_h.
 \end{equation*}
The vector potential $\vect{\psi}_h$ would be uniquely determined by adding the condition $\div\, \vp=0$. In fact, on a simply connected domain, ${\rm div}\, \vp=0$ and $\curl\,\vp=0$ together with
$\vp\in \vect{H}_{0,t}(\curl,\O)$ imply $\vp=0$.
However, in general, the solution of ${\rm div}\, \vp=0$ and $\curl\,\vp=\vect{v}_h$ together with
$\vp\in \vect{H}_{0,t}(\curl,\O)$ (which is uniquely determined) does not belong to $ \Nb$. A possibility to select a vector potential $\vect{\psi}_h$ in a unique way could be to compute it as the approximation to  the following continuous problem:
Find $(\vect{\psi},\theta)$ in $\vect{H}_{0,t} (\curl;\O) \times H^{1}_{0}(\O)$  such that
 \begin{equation*}\label{p1:0}
 \begin{aligned}
 (\curl\, \vp ,\curl\, \vect{\phi})_{\Th} +(\nabla \theta, \vect{\phi})_{\Th}&=(\bv_h ,\vect{\phi})_{\Th} \quad &\forall\, \vect{\phi} \in \vect{H}_{0,t} (\curl;\O),&&\\
 (\vp,\nabla s)_{\Th}&=0\quad &\forall\, s \in H^{1}_{0} (\O).&&
 \end{aligned}
 \end{equation*}
 Setting
 \begin{equation*}
\calWho:=\left\{ w \in H^{1}_0(\O)\,\, :\,\,  w |_{\K} \in \Pp^{k+1}(\K)\,\,
\forall\, \K\in \Th \right\},
\end{equation*}
the discrete problem reads: Find $(\vect{\psi}_h,\theta_h) \in \Nb\times \calWho$  such that
\begin{equation}\label{ppp00}
 \begin{aligned}
 (\curl\, {\vp}_h ,\curl\, \vect{\phi}_h)_{\Th} +(\nabla \theta_h, \vect{\phi}_h)_{\Th}&=(\bv_h ,\vect{\phi}_h)_{\Th} \quad &\forall\, \vect{\phi}_h \in \Nb,&&\\
 ({\vp}_h,\nabla w_h)_{\Th}&=0\quad &\forall\, w_h \in \calWho.&&
 \end{aligned}
 \end{equation}
 Problem \eqref{ppp00} has a unique solution satisfying $\curl \, \vect{\psi}_h=\bv_h$ (from the first equation), and  div$\vect{\psi}_h=0$ (from the second equation).
\end{remark}

\section{Preconditioner: Fictitious Space Lemma and Auxiliary Space
    Framework}\label{sec:4}

\subsection{Preconditioner for the semi-definite system}
Assume $V$ is a Hilbert space equipped with the norm $\|\cdot\|_V$ and
that $A: V\mapsto V'$ is a bounded linear operator.  We define the
bilinear form
$$
(u,v)_A=\langle Au,v\rangle.
$$

We say $A$
is symmetric if the bilinear form $(u,v)_A$ is symmetric.
We say that $A$ is  semi-positive definite if
$$
(v,v)_A \ge 0,\quad\forall v\in V
$$
and $\alpha>0$ exists such that
$$
(v,v)_A\geq \alpha\|v\|_{V/N(A)}^2,\quad \forall v\in V/N(A).
$$
And we say that $A$ is SPD (Symmetric Positve Definite) if it is symmetric and $\alpha>0$ exists such that
$$(v,v)_A\geq \alpha\|v\|_V^2,\quad\forall v\in V.$$
One useful property of symmetric semi-positive definite operators is that
\begin{equation}
  \label{vAv}
Av=0 \mbox{ iff } \langle Av, v\rangle =0.
\end{equation}
A preconditioner for $A$ is another symmetric semi-positive definite
operator $B: V'\mapsto V$.  Again, we consider the bilinear form
$$
(f,g)_B=\langle f, Bg\rangle.
$$
The operator $BA: V\mapsto V$ satisfies
$$
(BAu,v)_A=\langle Av, BAu\rangle=(Au,Av)_B.
$$
\begin{lemma} If $A:V\mapsto V'$ and $B:V'\mapsto V$ are both symmetric
semi-positive definite such that $B$ is positive definite on $R(A)$, then
\begin{enumerate}
\item $B: R(A)\mapsto R(BA)$ is an isomorphism (with the
  inverse satisfying trivially that $B^{-1}(BAv)=Av$).
\item The bilinear form $(\cdot,\cdot)_{B^{-1}}$ defines an inner product on $R(BA)$.
\item The bilinear form $(\cdot,\cdot)_A$ defines an inner product on
  $R(BA)$.
\item $BA$ is symmetric positive definite on $R(BA)$ with either of
  the above two inner products.
  \end{enumerate}
\end{lemma}
\begin{proof} All these results are pretty obvious, and their proofs are similar. Let
  us give the proof for 3 as an example.

  We only need to verify that $(\cdot,\cdot)_A$ is positive definite
  on $R(BA)$.  If $v\in R(BA)$ is such that $(v,v)_A=0$, then, by
  \eqref{vAv}, we have $Av=0$.  We write $v=BAw$ for some $w\in V$,
  then $ABAw=0$ and hence $(Aw,Aw)_B=0$.  As $B$ is positive definite
  on $R(A)$, we have $Aw=0$.  Thus, $v=ABAw=0$, as desired.
\end{proof}

For the system $Au=f$, we can apply the preconditioner $B$ and the
preconditioned conjugate gradient (PCG) method with respect to the inner
product $(\cdot,\cdot)_{B^{-1}}$ with the following convergence estimate:
$$
\|u-u^k\|_A\le 2\left(\frac{\sqrt{\kappa(BA)}-1}{\sqrt{\kappa(BA)}+1}\right)^k\|u-u^0\|_A.
$$
The condition number can then be estimated by $\kappa(BA) \le
c_1/c_0$, either where
$$
c_0(v,v)_{B^{-1}}\le (BAv,v)_{B^{-1}}\le c_1
(v,v)_{B^{-1}},\quad\forall v\in R(BA),
$$
or equivalently where
$$
c_0(w,w)_{B}\le (Bw,Bw)_{A}\le c_1 (w,w)_{B},
\quad\forall w\in R(A),
$$
or where
$$
c_1^{-1}(v,v)_{A}\le (B^{-1}v,v) \le c_0^{-1}(v,v)_{A}
\quad\forall v\in R(BA).
$$
\subsection{Fictitious space lemma and generalizations}
Let us present and prove a refined version of the Fictitious Space Lemma
originally proposed by Nepomnyaschikh \cite{NEP1991} (see also
\cite{JXU96}).
\begin{lemma}
\label{lm:auxiliary}
Let $\utilde{V}$ and $V$ be two Hilbert spaces, and let
$\Pi:\utilde{V}\mapsto V$ be a surjective map. Let
$\utilde{B}:\utilde{V'}\mapsto \utilde{V}$ be a symmetric and positive
definite operator. Then $B:=\Pi\utilde{B}\Pi'$ is also symmetric and
positive definite (here $\Pi': V'\mapsto\utilde{V}'$ is such that
$\langle\Pi'g,\utilde v\rangle=\langle g,\Pi\utilde v\rangle$,
for all $g\in V'$ and $\utilde v\in \utilde V$).
Furthermore,
 $$
\langle  B^{-1}v,v\rangle =\inf_{\Pi\utilde{v}=v}\langle
\utilde{B}^{-1}\utilde{v},\utilde{v}\rangle.
$$
\end{lemma}

\begin{proof} It is obvious that $B$ is symmetric and positive
  semi-definite.  Note that if $v\in V^\prime$ is such that $\langle Bv,v\rangle
  =0$, then $\langle \utilde{B}\Pi'v,\Pi'v\rangle = \langle
  Bv,v\rangle =0$. This means that $\Pi'v=0$ as $\utilde B$ is
  SPD. Hence, $v=0$ as $\Pi'$ is injective.  This proves that $B$ is
  positive definite.

  For any $\utilde{v}\in \utilde{V}$, let $v=\Pi\utilde{v}$ and
  $\utilde{v}^*=\utilde{B}\Pi'B^{-1}v$.  As we obviously have $\Pi
 \utilde{v}^*=v$, we can write $\utilde{v}=\utilde{v}^*+\utilde{w}$ with
  $\Pi\utilde{w}=0$.  Thus,
\begin{equation*}
\begin{split}
\inf_{\Pi\utilde{v}=v}\langle \utilde{B}^{-1}\utilde{v},\utilde{v}\rangle
=&\inf_{\Pi\utilde{w}=0}\langle \utilde{B}^{-1}( \utilde{v}^*+\utilde{w}) ,\utilde{v}^*+\utilde{w}\rangle \\
=&\langle \utilde{B}^{-1}\utilde{v}^*,\utilde{v}^*\rangle +\inf_{\Pi\utilde{w}=0}\left( \langle \utilde{B}^{-1}\utilde{w},\utilde{w}\rangle +2\langle \utilde{B}^{-1}\utilde{v}^*,\utilde{w}\rangle \right)
\end{split}\end{equation*}
From the definition of $\utilde{v}^*$ we have
$$
\langle \utilde{B}^{-1}\utilde{v}^*,\utilde{v}^*\rangle =
\langle B^{-1}v,\Pi \utilde{v}^*\rangle =\langle B^{-1}v,v\rangle,
$$
and also
$$
\langle \utilde{B}^{-1}\utilde{v}^*,\utilde{w}\rangle =\langle
\utilde{B}^{-1}\utilde{B}\Pi'B^{-1}v,
\utilde{w}\rangle =\langle \Pi'B^{-1}v,\utilde{w}\rangle
=\langle B^{-1}v,\Pi\utilde{w}\rangle =0.
$$
The last two identities lead to the desired result.
\end{proof}
\begin{theorem}
\label{thm:auxiliary}
Assume that $\utilde A: \utilde V\mapsto \utilde V'$ and $A: V\mapsto
V'$ are symmetric semi-definite operators.  We assume that $\Pi:
\utilde V\mapsto V$ is surjective and that $\Pi(N(\tilde A))=N(A)$.  Then
for any SPD operator $\utilde B:\utilde V'\mapsto \utilde V$, we have,
for $B=\Pi\utilde B\Pi'$,
$$
\kappa (BA)\le \kappa(\Pi)\kappa(\utilde B\utilde A).
$$
Here $\kappa(\Pi)$ is the smallest ratio $c_1/c_0$ that satisfies
\begin{equation}
  \label{Pi}
c_1^{-1}\langle Av,v\rangle \le\inf_{\Pi\utilde v=v}\langle \tilde A\tilde v,\tilde v\rangle \le c_0^{-1}\langle Av,v\rangle,\quad \forall v\in R(BA).
\end{equation}
\end{theorem}
\begin{proof}
Denote $\kappa(\utilde B\utilde A)=b_1/b_0$ with $b_1$ and $b_0$ satisfying
$$
b_1^{-1}(\utilde v, \utilde v)_{\utilde A}\le
(\utilde B^{-1}\utilde v, \utilde v)
\le b_0^{-1} (\utilde v, \utilde v)_{\utilde A},\quad\forall \utilde
v\in R(\tilde B\tilde A).
$$
By \eqref{Pi}, we obtain
$$
b_1^{-1}c_1^{-1}\|v\|_A^2\le \inf_{\Pi\utilde v=v,\tilde v\in R(\tilde
  B\tilde A)} (\utilde B^{-1}\utilde v, \utilde v) \le b_0^{-1}
c_0^{-1}\|v\|_A^2,\quad \forall v\in R(BA).
$$
By the assumption that $\Pi(N(\tilde A))=N(A)$, we can prove that
$\Pi'(R(A))\subset R(\tilde A)$ and


$$
\{\tilde v| \Pi\tilde v=v\in R(BA)\}=\{\tilde v| \Pi\tilde v=v\in
R(BA), \tilde v\in R(\tilde B\tilde A)\}.
$$
By Lemma \ref{lm:auxiliary},
$$\inf_{\Pi \tilde v=v, \tilde v\in R(\tilde B\tilde A)}(\tilde B^{-1}\tilde v,\tilde v)=\inf_{\Pi \tilde v=v}(\tilde B^{-1}\tilde v,\tilde v)=(B^{-1}v,v),\quad \forall v\in R(BA).$$
Therefore,
$$b_1^{-1}c_1^{-1}\|v\|_A^2\leq(B^{-1}v,v)\leq b_0^{-1}c_0^{-1}\|v\|_A^2\quad \forall v\in R(BA).$$
\end{proof}
\begin{theorem}\label{thm:BA}
Assume that the following two conditions are satisfied for $\Pi$.  First,
$$
\|\Pi \utilde v\|_A\le c_1\|\utilde v\|_{\utilde A}, \quad\forall
\utilde v\in \utilde V.
$$
Second, for any $v\in V$ there exists $\utilde v \in\utilde V$ such
that $\Pi\utilde v=v$ and
$$
\|\utilde v\|_{\utilde A}\le c_0\|v\|_A.
$$
Then $\kappa(\Pi)\le c_1/c_0$ and, under the assumptions of Theorem \ref{thm:auxiliary},
$$
\kappa (BA)\le \left(\frac{c_1}{c_0}\right)^2\kappa(\utilde B\utilde A).
$$
\end{theorem}

\begin{remark}\label{general} In view of the application of the above
  results to our two dimensional case (as we shall see in the next subsection), it
  would have been enough to restrict ourselves to the symmetric
  positive definite case (instead of the semi-definite case treated in
  the last two subsections). However we preferred to have them in the
  present more general setting, as in this form they are likely to be
  useful in many other circumstances (starting, as natural, from the
  extension of the present theory to the three-dimensional case).
\end{remark}

\subsection{Application to our problem}\label{aux-section}
In this section we design a simple preconditioner for the linear
system resulting from the approximation of the Stokes problem
\eqref{var:0} defined in \eqref{dg:bdm0}-\eqref{method:0}.  Note that
the bilinear form $a_h(\cdot,\cdot)$ defined in \eqref{method:0}
provides a discretization of the vector Laplacian problem
\begin{equation*}
-\bdiv (2\nu  \Gsym{u} ) =\f \quad \mbox{in} ~\Omega, \quad \bu\cdot\n =0,\;\;(\Gsym{u}\cdot \n)\cdot \vect{t}
=0 \mbox{   on  } \Gamma .
\end{equation*}
We denote by $A_h$ the operator associated with $a_h(\cdot,\cdot)$.
As the solution $\bu_h \in \Velho$ of \eqref{dg:bdm0} is
divergence-free, the discrete Helmholtz decomposition
\eqref{decomp-helm2} implies that
\begin{equation*}
\mbox{a unique } \,\psi_h\in \Nho  \quad \mbox{exists such that } ~\bu_h=
  \curl \,\psi_h.
\end{equation*}
At this point, it is convenient to introduce the space $\Velhoo$ as
\begin{equation} \label{defVelhoo}
\Velhoo:=\Velho\cap \vect{H}_0(\div^{0};\O).
\end{equation}
We note that as the sequence \eqref{seqex2d} is exact, we have
\begin{equation} \label{Velhooeqcurl}
\Velhoo\equiv\curl\,\Nho,
\end{equation}
and that the mapping is one-to-one.  Therefore, restricting the bilinear
form $a_h(\cdot,\cdot)$ to $\Velhoo$, in the spirit of Remark \ref{sudiv0},
corresponds here to restricting the trial and test space to
$\Velhoo\equiv\curl (\Nho)$. The discrete problem \eqref{dg:bdm0} then
reduces to the following problem:
{\it Find $\psi_h\in\Velhoo$ such that}
\begin{equation}\label{V0A}
a_h(\psi_h, \varphi_h)=(\f, \varphi_h) \qquad \forall\, \varphi_h \in \Velhoo
\end{equation}
Defining the operator $A_h:\Velhoo\mapsto\Velhoo'$ by $\langle
A_h\psi_h, \varphi_h\rangle =a_h(\psi_h, \varphi_h), \psi_h,
\varphi_h\in \Velhoo$, we can write \eqref{V0A} as
$$
A_h\psi_h=f_h.
$$
We now use the original space $\Velho$ as the auxiliary space for
$\Velhoo$.  Define $\utilde A_h: \Velho\mapsto \Velho'$ by
$\langle \utilde A_hu_h, v_h\rangle =a_h(u_h, v_h), u_h,v_h\in
\Velho$.   We note that $\utilde A_h$ is a discrete Laplacian.  We
assume that $\utilde B_h$ is an optimal preconditioner for $\utilde A_h$.

We now define the  operator
\begin{equation}\label{pi}
\Pi_h : \Velho \lor \quad\Velhoo\;\equiv\;
\curl (\Nho)
\end{equation}
according to \eqref{decomp-helm2}, namely
$$
\Pi_h\bv_h= \curl \,\varphi_h.
$$
Note that $\Pi_h$ is
a surjective operator and that $\Pi_h$  acts as the identity on the
subspace $ \Velhoo$.  The auxiliary space preconditioner for $A_h$ is
then defined by
\begin{equation}\label{Bh}
B_h=\Pi_h\utilde B_h\Pi_h^*.
\end{equation}

\begin{lemma}\label{stab:Pi}
  Assume that the spaces$ (\Velho, \calQho, \Nho)$ satisfy
  assumption ${\bf H0}$.  Then $B_h$ given by \eqref{Bh} is an optimal
  preconditioner for $A_h$ as long as $\utilde B_h$ is an optimal
  preconditioner for $\utilde A_h$.
\end{lemma}

\begin{proof}
  Following the auxiliary space techniques (Theorem~\ref{thm:BA}), we need to check that the
  following two properties are satisfied:
\begin{description}
\item[{\bf (A1)}] Local Stability: there exists a positive constant
  $C_1$ independent of $h$ such that
\begin{equation}\label{a11}
\|\Pi_h \bv_h\|_{DG} \leq C_1 \|\bv_h\|_{DG} \quad \forall\, \bv_h \in  \Velho
\end{equation}
\item[{\bf (A2)}] Stable decomposition: there exists a positive
  constant $C_2$ independent of $h$ such that for any
  $\bw_h\in\Velhoo$ there exists $\bv_h\in\Velho$ such that
  $\Pi_h\bv_h=\bw_h$ and
\begin{equation}\label{a22}
 \|\bv_h\|_{DG}\leq C_2 \|\bw_h\|_{DG}.
\end{equation}
\end{description}
To prove \eqref{a11} from the Helmholtz decomposition
\eqref{decomp-helm2} and the definition \eqref{pi} of $\Pi_h$, we have
\begin{equation}\label{splitbvh}
\bv_h=\gradh q_h+ \curl \,\varphi_h =\gradh q_h+\Pi_h\bv_h.
\end{equation}
Using estimate \eqref{est:s-v0} from Lemma \ref{helmholtz} and the
clear fact that $\div\,\bvh$ is the trace of $\Gsym{\bvh}$, we have
\begin{equation}\label{boundghbvh}
\|\gradh q_h\|_{DG}\leq C\|\div \,\bv_h\|_{0,\O} \leq C\| \Gsym{v}\|_{0,\Th}\leq C\|\bv_h\|_{DG} .
\end{equation}
Hence, \eqref{a11} follows from \eqref{splitbvh} and \eqref{boundghbvh}:
\begin{equation*}
\|\Pi_h \bv_h\|_{DG}=\|\bv_h -\gradh q_h\|_{DG}\leq \|\bv_h\|_{DG}+\|\gradh q_h\|_{DG}\leq C \|\bv_h\|_{DG} .
\end{equation*}
Finally, the inequality~\eqref{a22} holds with $C_2=1$ by taking $\bvh=\bw_h$.
\end{proof}

\section{Numerical experiments}\label{sec:5}

\subsection{Setup}
The tests presented in this section use discretization by the lowest
order, namely, $\BDM_{1}$ elements paired with piece-wise constant space for the
pressure. They verify the \emph{a priori} estimates given in
Theorem~\ref{teo0} and confirm the uniform bound on the condition
number of the preconditioned system for the velocity.

As previously set up, the discrete problem under consideration is given
by equation~\eqref{dg:bdm0} with bilinear forms $a_h(\cdot,\cdot)$ and
$b(\cdot,\cdot)$ defined in~\eqref{method:0}.  In the numerical
tests presented here, we take $\nu=1/2$ and the penalty parameter $\alpha=6$
in~\eqref{method:0}.
\begin{figure}[!ht]
  \subfloat[Coarsest mesh]{\includegraphics*[width=0.35\textwidth]{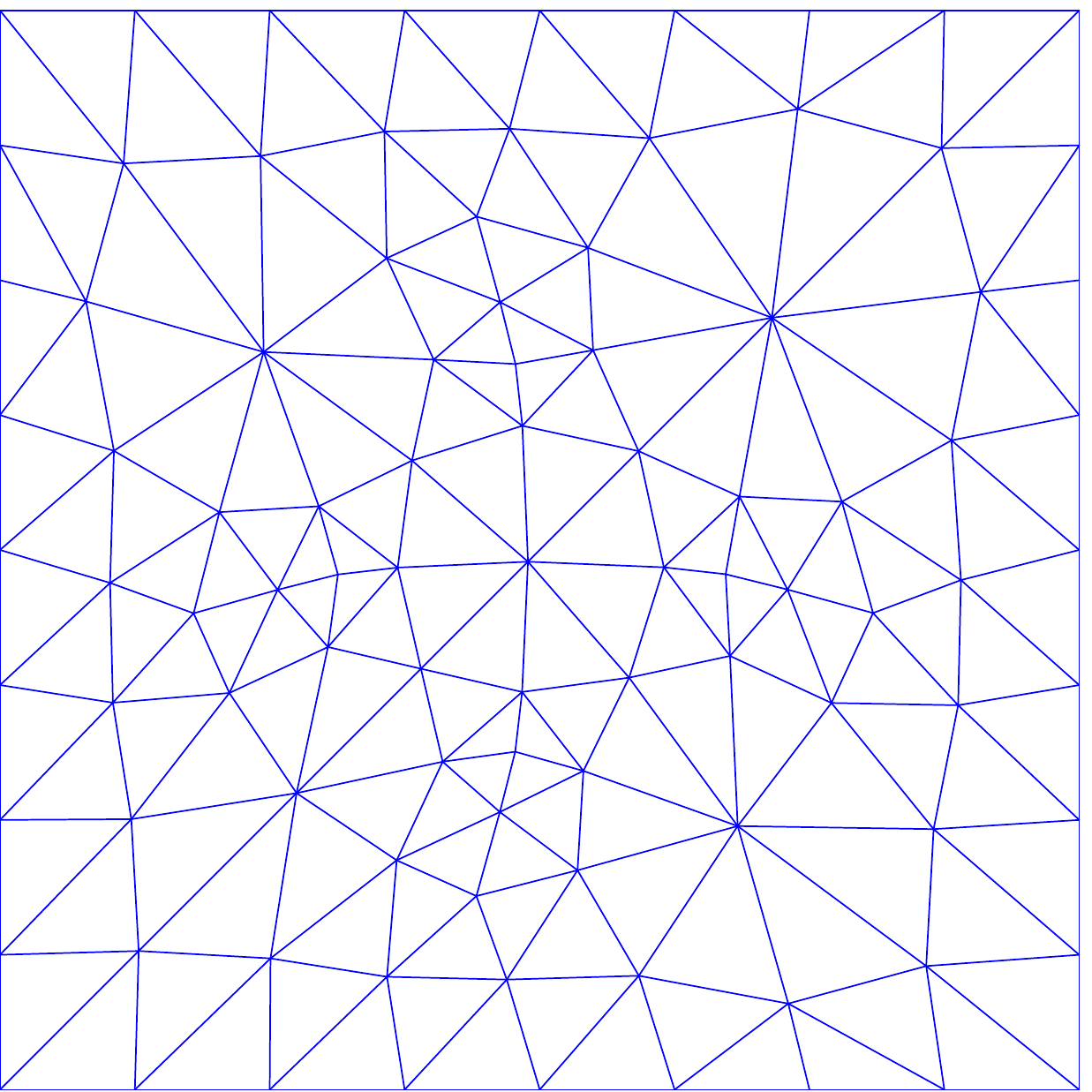}}
  \hspace*{0.1\textwidth} \subfloat[Mesh for level of refinement
  $J=3$]{\includegraphics*[width=0.35\textwidth]{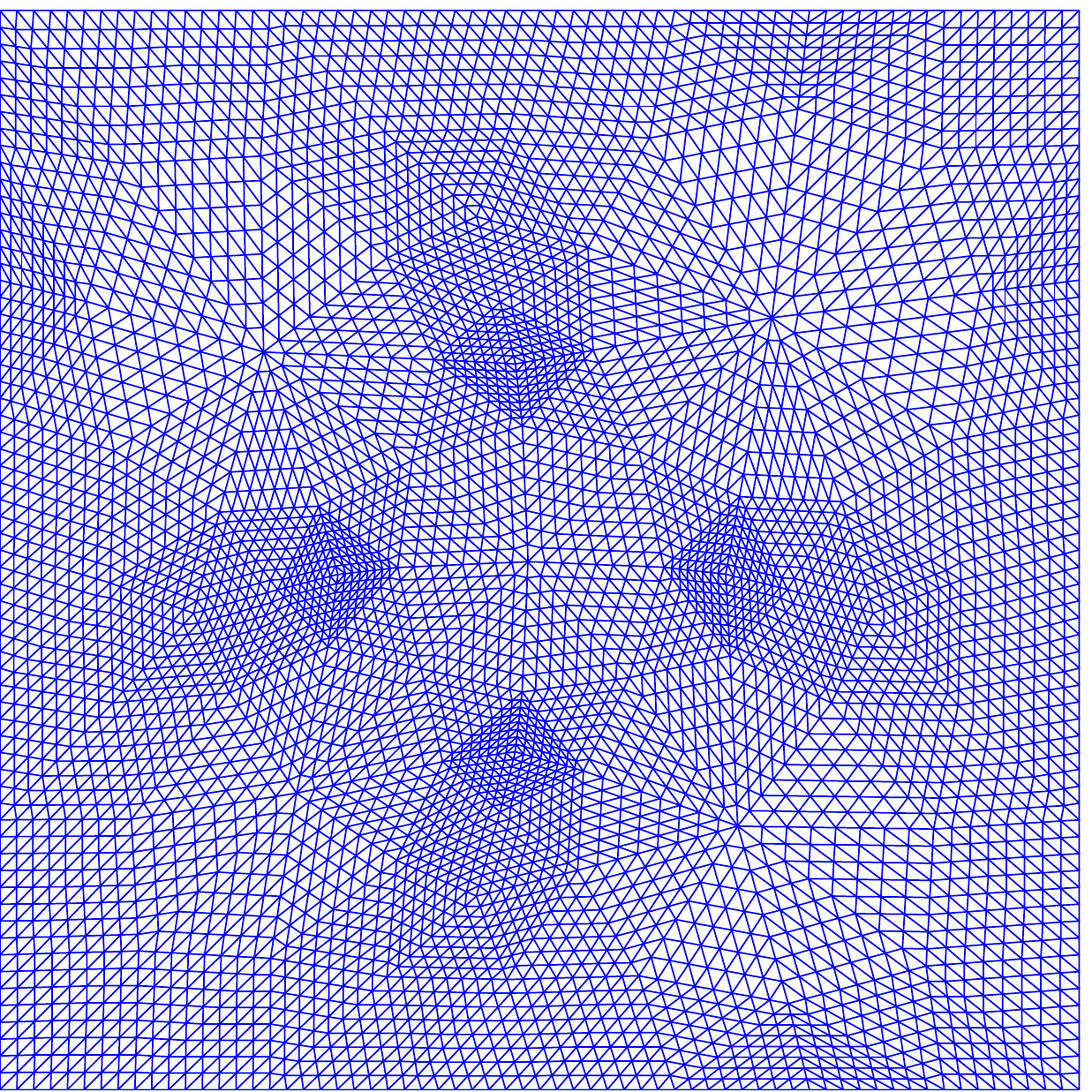}}
\caption{Meshes used in the tests for the unit square domain\\
$\Omega=(0,1)\times(0,1)$\label{fig:sq0and3}}
\end{figure}
We present two sets of tests with $A$ corresponding to the Stokes
equation discretized on a sequence of successively refined
unstructured meshes as shown in
Figures~\ref{fig:sq0and3}--\ref{fig:L0and3}.  On the square the
coarsest mesh (level of refinement $J=0$) has $160$ elements and $97$
vertices with $448$ BDM degrees of freedom. The finer triangulations
of the square domain are obtained via $1,\ldots,5$ regular refinements
(every element divided in $4$) and the finest one is with $163,840$
elements, $82,433$ vertices and $490,496$ BDM degrees of freedom.
Similarly for the $L$-shaped domain we start with a coarsest grid
($J=0$) with $64$ vertices and $97$ elements. For the $L$-shaped
domain the finest grid (for $J=5$) has $99,328$ elements, $50,129$ vertices
and $297,056$ $\BDM_1$ degrees of freedom.
\begin{figure}[!ht]
  \subfloat[Coarsest mesh]{\includegraphics*[width=0.35\textwidth]{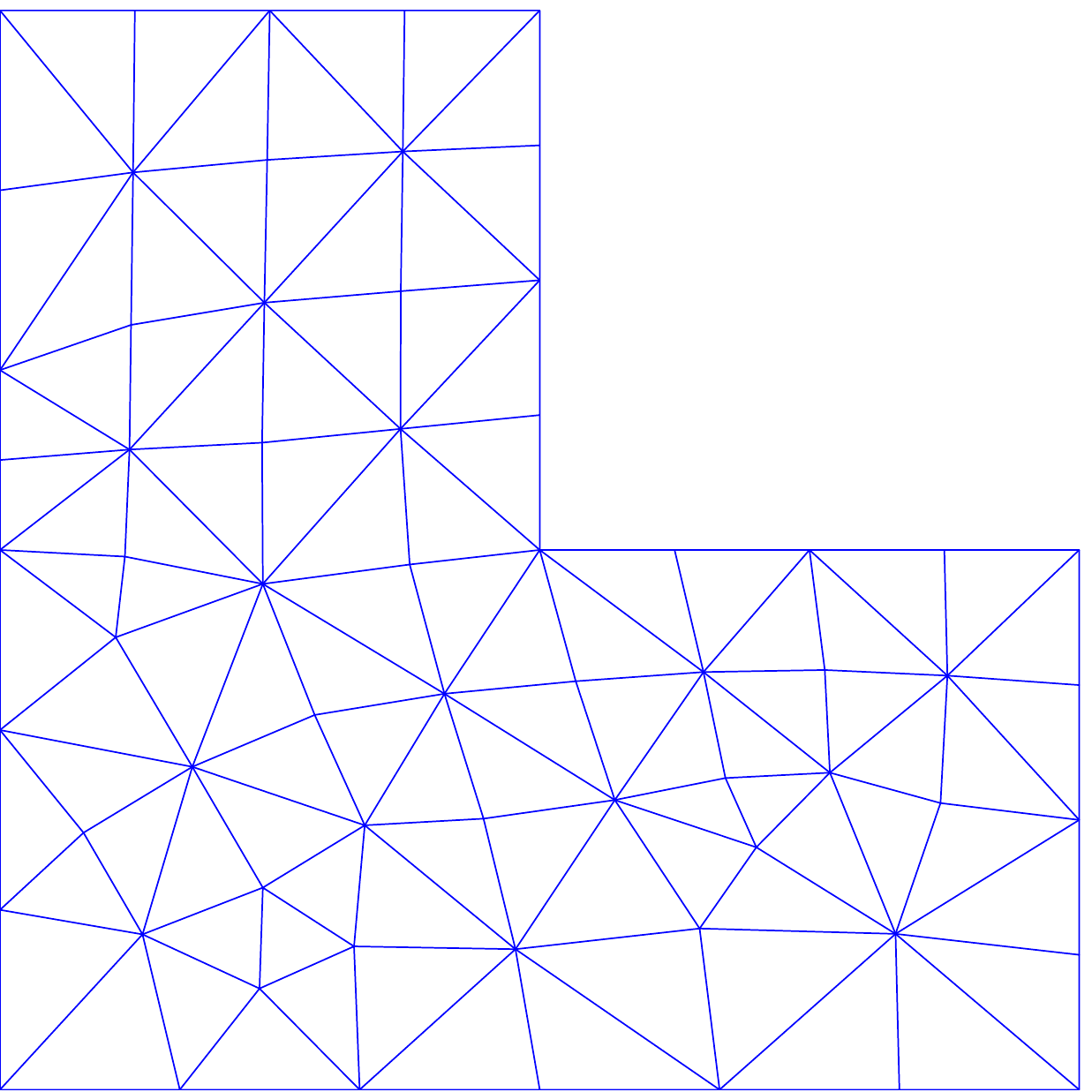}}
  \hspace*{0.1\textwidth} \subfloat[Mesh for level of refinement
  $J=3$]{\includegraphics*[width=0.35\textwidth]{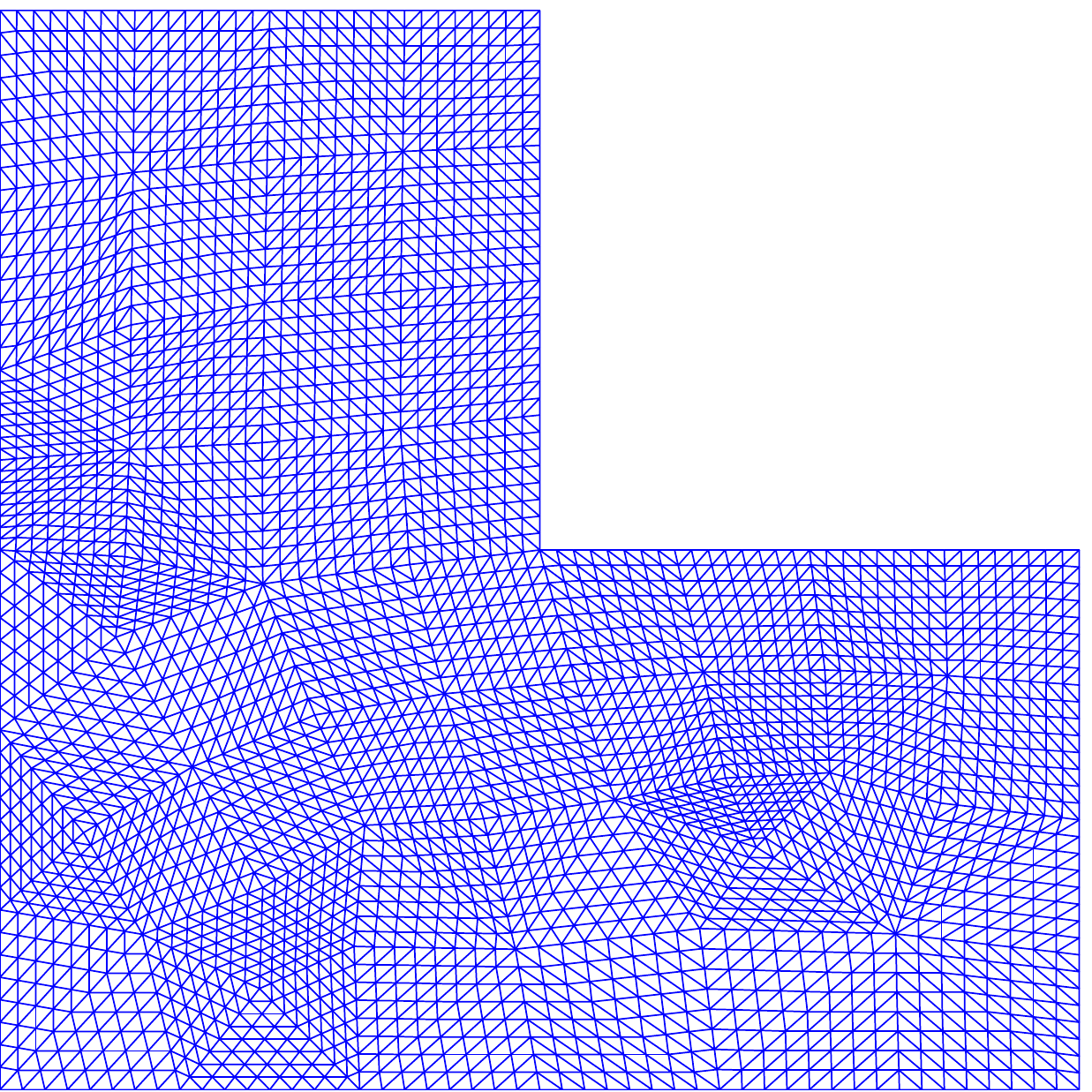}}
\caption{Meshes used in the tests for the $L$-shaped domain\\
  $\Omega=((0,1)\times (0,1))\setminus([\frac12,1)\times [\frac12,1))$\label{fig:L0and3}}
\end{figure}
In the computations, we approximate the velocity component $\bu_h$ of
the solution of the Stokes equation by solving several simpler
equations (such as scalar Laplace equations). After we obtain the
velocity, the pressure then is found via a postprocessing step at low
computational cost.  Further, for this sequence of grids the $\BDM_1$
interpolant of a function $\bv$ on the $k-th$ grid is denoted by
$\bv^{I_k}$. Accordingly the piece-wise constant, $L_2$-orthogonal
projection of $p$ is denoted by $p^{I_k}$.  We also use the notation
$(\bu_k,p_k)$ for the solution of~\eqref{dg:bdm0} on the $k-th$ grid, $k=0,\ldots,5$.
 \subsection{Discretization error} We now present several tests
 related to the error estimates given in the previous sections.
 We computed and tabulated  approximations of the order of
 convergence of the discrete solution in different norms. These
 approximations are denoted by
 $\gamma_0\approx\beta_0$, $\gamma_{DG}\approx\beta_{DG}$, 
 $\gamma_p\approx\beta_p$, and $\gamma_{*}\approx\beta_{*}$. The
 actual orders of convergence  $\beta_0$, $\beta_{DG}$, $\beta_{p}$,
 and $\beta_{*}$ are
\begin{equation*}
\begin{array}{l}
\|\bu-\bu_h\|_{0,\Omega}\approx C(\bu)h^{\beta_0},\quad
\|\bu-\bu_h\|_{DG}\approx C(\bu)h^{\beta_{DG}}, \\
\|p-p_h\|_{0,\Omega}\approx C(\bu,p)h^{\beta_p},\quad
|\jump{\bu_h}|_{*}\approx C(\bu)h^{\beta_{*}}.
\end{array}
\end{equation*}
Here, as in~\eqref{eq:jump-semi-norm}, we denote
$$|\jump{\bv}|^2_{*} = \sum_{e\in\Eho} h_e^{-1}\int_e\jump{\bu_t}^2\,ds.$$
Note that $\beta_*$ is the order with which the
jumps in the approximate solution (not in the error) go to zero.

We present two sets of experiments to illustrate the results given
in Theorem~\ref{teo0}. First, we consider the exact given solution and
calculate the right--hand side and the boundary conditions from this
solution. We set
\begin{equation}\label{eq:uexact}
\phi=xy(1-x)(2x-1)(y-1)(2y-1),\qquad \bu=\curl\phi.
\end{equation}
Clearly, the function $\phi$ vanishes on the boundary of both the
domains under consideration and we take $\bu$ defined
in~\eqref{eq:uexact} as exact solution for the velocity for both the
square and the $L$-shaped domains.  For the pressure we choose as
exact solutions functions with zero mean value and select $p$ different
for the square and the $L$-shaped domain, namely
\begin{equation}\label{eq:pexact}
\begin{array}{ll}
p=x^2-3y^2+\frac83 xy, &\mbox{(square domain),}\\
p=x^2-3y^2+\frac{24}{7}xy, &\mbox{($L$-shaped domain)}.
\end{array}
\end{equation}
The right hand side $\f$ is calculated by plugging $(\bu,p)$ defined
in~\eqref{eq:uexact}--\eqref{eq:pexact}
in~\eqref{Stokes:F}. Table~\ref{table:uh-minus-ui-sq} shows tabulation of
the order of convergence of $(\bu_h,p_h)$ to $(\bu^I,p^I)$ for both
the square domain and the $L$-shaped domain. The values approximating
the order of convergence displayed in
Table~\ref{table:uh-minus-ui-sq} are
\begin{eqnarray*}
&&\gamma = \log_2\frac{\|\bu^{I}_{k-1}-\bu_{k-1}\|}{\|\bu^{I}_{k}-\bu_{k}\|},\quad
\gamma_* = \log_2\frac{|\jump{\bu_{k}}|_*}{|\jump{\bu_{k-1}}|_*},\\
&&\gamma_p=\log_2
\frac{\|p^{I}_{k-1}-p_{k-1}\|_{0,\Omega}}{\|p^{I}_{k}-p_{k}\|_{0,\Omega}},
\quad k=1,\ldots,5.
\end{eqnarray*}
Here $\|\cdot\|$ stands for any of the $DG$ or $L_2$ norms. The
quantity $\gamma$ is the corresponding $\gamma_0$ or $\gamma_{DG}$.
From the results in this table, we can conclude that in
the $\|\cdot\|_{DG}$ norm the dominating error is the interpolation
error, and as the next example shows, in general, the order of
convergence in $\|\cdot\|_{DG}$ is $1$.
\begin{table}
\caption{Approximate order of convergence  for the difference
$(\bu^{I}-\bu_h)$ and $(p^{I}-p_h)$ and the jumps $|\jump{\bu_h}|_*$
for the square and $L$-shaped  domains. Here, $\bu$ and $p$ are given in~\eqref{eq:uexact} and \eqref{eq:pexact}.\label{table:uh-minus-ui-sq}}
\begin{tabular}{||c|c|c|c|c|c||}
\hline
\hline
\multicolumn{6}{||c||}{Square domain}\\
\hline
$k$ &  $1$ & $2$ & $3$ & $4$ & $5$\\
\hline
$\gamma_0$    & 1.75 & 1.87 & 1.94& 1.98 & 1.99 \\
\hline
$\gamma_{DG}$ & 0.98 & 1.0 & 1.00 & 1.00& 1.00 \\
\hline
$\gamma_p$    & 0.94& 0.95& 0.97& 0.99& 0.99 \\
\hline
$\gamma_*$    & 0.77& 0.89& 0.95& 0.98& 0.99 \\
\hline
\end{tabular}
\begin{tabular}{||c|c|c|c|c|c||}
\hline
\hline
\multicolumn{6}{||c||}{L-shaped domain}\\
\hline
$k$ &  $1$ & $2$ & $3$ & $4$ & $5$ \\
\hline
$\gamma_0$    & 1.69 & 1.79& 1.90 & 1.96& 1.98 \\
\hline
$\gamma_{DG}$ &0.97 & 1.01&1.01 & 1.00 & 1.00 \\
\hline
$\gamma_p$    &0.93 & 0.92 &0.95 & 0.97& 0.99\\
\hline
$\gamma_*$    &0.73& 0.85& 0.93& 0.97& 0.99 \\
\hline
\end{tabular}
\end{table}

The second test is for a fixed right hand side $\f=2(1,x)$. We
calculate approximations to the order of convergence of the numerical
solutions on successively refined grids as follows:
\begin{eqnarray*}
&&\gamma=\log_2\frac{\|\bu_{k}-\bu_{k-1}\|}{\|\bu_{k+1}-\bu_{k}\|},
\quad
\gamma_*=\log_2\frac{|\jump{\bu_{k}}|_*-|\jump{\bu_{k-1}}|_*}{|\jump{\bu_{k+1}}|_*-|\jump{\bu_{k}}|_*},\\
&&\gamma_p=\log_2\frac{\|p_{k}-p_{k-1}\|_{0,\Omega}}{\|p_{k+1}-p_{k}\|_{0,\Omega}},\quad
k=1,\ldots,4.
\end{eqnarray*}
Again, $\|\cdot\|$ denotes any of the (semi)-norms of interest and $\gamma$
approximates the corresponding order of convergence.
Table~\ref{table:uh-minus-ui-L}  shows the tabulated values of $\gamma_0$,
$\gamma_{DG}$, $\gamma_p$, and $\gamma_{*}$. It is
clear from these values that the order of approximation for the velocity and the
pressure is optimal for the square domain, whereas for the $L$-shaped
domain the convergence is not of optimal order, due to the singularity of
the solution near the reentrant corner.
\begin{table}
\caption{Approximate order of convergence of the error for square and
  $L$-shaped domains and right--hand side $\f=2(1,x)$. \label{table:uh-minus-ui-L} }
\begin{tabular}{||c|c|c|c|c|c||}
\hline
\hline
\multicolumn{6}{||c||}{Square domain}\\
\hline
$k$ &  $1$ & $2$ & $3$ & $4$ & $5$\\
\hline
$\gamma_0$    &1.70&1.85 &1.93& 1.97  & 1.98\\
\hline
$\gamma_{DG}$ &0.86 &0.95&0.98& 0.99& 1.00\\
\hline
$\gamma_p$    &0.94&0.94&0.97& 0.98  & 0.99\\
\hline
$\gamma_*$    &0.70&0.86&0.94& 0.97  & 0.99\\
\hline
\end{tabular}
\begin{tabular}{||c|c|c|c|c|c||}
\hline
\hline
\multicolumn{6}{||c||}{L-shaped domain}\\
\hline
$k$ &  $1$ & $2$ & $3$ & $4$ & $5$ \\
\hline
$\gamma_0$    & 1.65 & 1.79& 1.86 & 1.74  & 1.24 \\
\hline
$\gamma_{DG}$ & 0.84&0.92& 0.92 & 0.86 & 0.74\\
\hline
$\gamma_p$    &0.91& 0.89& 0.88& 0.82 &  0.70\\
\hline
$\gamma_*$    &0.63& 0.81&0.89& 0.89  & 0.83\\
\hline
\end{tabular}
\end{table}
  The numerical experiments and also the approximations for the orders
  of convergence presented in Table~\ref{table:uh-minus-ui-sq} and
  Table~\ref{table:uh-minus-ui-L} are computed using the FEniCS package~\url{http://fenicsproject.org}.

 \subsection{Uniform preconditioning}
 The tests presented in this subsection illustrate the efficient
 solution of the system~\eqref{eq:matrix} below by Preconditioned
 Conjugate Gradient (PCG) with the preconditioner given
 in~\eqref{eq:preconditioner}. We introduce the matrices representing
 the bilinear forms defined in~\eqref{dg:bdm0}--\eqref{method:0}, and
 also the mass matrix for the $\BDM_1$ space. We denote by
 $\mathbf{M}$ the mass matrix on $\Velho$ and by
 $\widetilde{\mathbf{A}}$ the stiffness matrix associated with
 $a_h(\cdot,\cdot)$ on $\Velho$ in~\eqref{dg:bdm0}--\eqref{method:0}.
 We note that $\mathbf{A}$, without the divergence--free constraint,
 is spectrally equivalent to two scalar Laplacians.

 It is known that the null space of $b(\cdot,\cdot)$
 in~\eqref{dg:bdm0} is made of vector fields that are curls of
 continuous, piecewise quadratic functions vanishing on the boundary.
 We denote by $\mathbf{P}_{\text{curl}}$ the matrix representation of these curls in the BDM space.  Namely,
$$
\curl(\mbox{basis functions in $\Nho$})=
 (\mbox{basis functions in $\Velho$})\mathbf{P}_{\text{curl}}.
$$
It is easy to see that
$$
\mathbf{A}_q = \mathbf{P}_{\text{curl}}^T \mathbf{M} \mathbf{P}_{\text{curl}}.
$$
where $\mathbf{A}_q$ is the discretization of the Laplacian on $N_h$
with homogeneous Dirichlet boundary conditions.

The problem of finding the solution of \eqref{V0A} then amounts to
solving the following algebraic system of equations
\begin{equation}\label{eq:matrix}
\mathbf{P}_{\text{curl}}^T\widetilde{\mathbf{A}} \mathbf{P}_{\text{curl}}
\mathbf{U} = \mathbf{P}_{\text{curl}}^T \mathbf{F}.
\end{equation}
Here the superscript $T$ means that the adjoint is taken with respect to the
$\ell_2$-inner product, $\mathbf{U}$ is the vector containing the
velocity degrees of freedom, and $\mathbf{F}$ is the vector
representing the right--hand side $(\f,\bv)$ of the
problem~\eqref{dg:bdm0}.

The matrix representation $\mathbf{B}$  of the preconditioner $B$
described in the previous section has the following
form:
\begin{equation}\label{eq:preconditioner}
 \mathbf{B}=\mathbf{A}_q^{-1} \mathbf{P}_{\text{curl}}^T
 \mathbf{M} \widetilde{\mathbf{A}}^{-1} \mathbf{M}\mathbf{P}_{\text{curl}}\mathbf{A}_q^{-1}
\end{equation}
In the numerical experiments below we have used the
  preconditioned conjugate gradient provided by MATLAB with the above
  preconditioner.
  We note that one may further make the algorithm more efficient by
  incorporating  approximations $\widetilde{\mathbf{B}}$ (for
  $\widetilde{\mathbf{A}}^{-1}$) and
$B_q$ (for $A_q^{-1}$) in~\eqref{eq:preconditioner}.  In our tests the inverses needed to compute the action of the
  preconditoner, namely $A_q^{-1}$ and $\widetilde{\mathbf{A}}^{-1}$,
  are calculated by the MATLAB's backslash "$\backslash$"
operator (which in turn calls
  the direct solver from
  UMFPACK~\url{http://www.cise.ufl.edu/research/sparse/umfpack/}). The
  tests presented here exactly match the theory for the auxiliary
  space preconditioner given in~Section~\ref{aux-section}.

  In summary, the action of the preconditioner requires the solution
  of systems corresponding to $4$ scalar Laplacians. It is also worth
  noting that suitable multigrid packages for performing these tasks
  are available today.

The convergence rate results are summarized in
Table~\ref{table:sqL}. The legend for the symbols used in the table is
as follows: $n_{it}$ is the number of PCG iterations; $\rho$ is the
average reduction per one such iteration defined as
$\rho=\left[\frac{||r_{n_{it}}||_{\ell{2}}}{||r_{0}||_{\ell{2}}}\right]^{1/{n_{it}}}$;
$J$ is the refinement level, for which $h\approx 2^{-J}h_0$, where
$h_0$ is the characteristic mesh size on the coarsest grid.
\begin{table}[!htb]
\caption{Preconditioning results for square domain (top) and $L$-shaped domain
  (bottom). The PCG
iterations are terminated when the relative residual is smaller than
$10^{-6}$.  \label{table:sqL}}
\begin{tabular}{||c|c|c|c|c|c|c||}
\hline
\hline
\multicolumn{7}{||c||}{Square domain}\\
\hline
J & 0 & 1 & 2 & 3 & 4 & 5\\
\hline
$n_{it}$    & 4 & 4 & 4 & 5 & 5 & 4 \\
\hline
$\rho$ & 0.016 & 0.023 & 0.031 & 0.034 & 0.033 & 0.031  \\
\hline
\end{tabular}
\begin{tabular}{||c|c|c|c|c|c|c||}
\hline
\multicolumn{7}{||c||}{L-shaped domain}\\
\hline
J & 0 & 1 & 2 & 3 & 4 & 5\\
\hline
$n_{it}$    & 5 & 5 & 5 & 5 & 5 & 5 \\
\hline
$\rho$ & 0.044 & 0.061 & 0.061 & 0.058 & 0.055 & 0.053 \\
\hline
\end{tabular}
\end{table}
From the results in Table~\ref{table:sqL}, we can conclude that the
preconditioner is uniform with respect
to the mesh size. It is also evident that this method
is in fact quite efficient in terms of
the number of iterations and the reduction factor.

Let us point out that when the preconditioner is implemented in 3D the
action of $\Pi_h$ requires an implementation of the action of
$L^2$-orthogonal (or orthogonal in equivalent inner product)
projection on the divergence free subspace $\Velhoo$.  This is done by
solving an auxiliary mixed FE discretization of the Laplacian, as
discussed in Section~\ref{ottob0} and in practice it can be
accomplished by considering a projection orthogonal in the inner
product provided by the lumped mass matrix for BDM. In such case the
solution to the auxiliary mixed FE problem corresponds to a solution
of a system with an $M$-matrix and
classical AMG methods~\cite{1982BrandtA_McCormickS_RugeJ-aa} AMG yield
optimal solvers for such problems.  The application of the
preconditioner in the 3D case requires the (approximate) solution
of 5 scalar Laplacians.

Such extensions to 3D and also efficient approximations to
$\widetilde{\mathbf{A}}^{-1}$ and $A_q^{-1}$
in~\eqref{eq:preconditioner} are subject of current research and
implementation and are to be included in a future release of the Fast Auxiliary
Preconditioning Package~\url{http://fasp.sf.net}.



\appendix
\section{Proof of Proposition \ref{perappA}}\label{app:A}
We now state and prove a result,  Proposition \ref{perappA:D} given below,
used in Section \ref{sec:2} to show Korn inequality (cf. Lemma \ref{korn}).
After giving its proof, we comment briefly on how the result can be applied to show the corresponding Korn inequality \eqref{korn1} (cf. Lemma \ref{korn}) for $d=3$.


\begin{proposition}\label{perappA:D}
Let $\K$ be a triangle  (or a tetrahedron for $d=3$)  with minimum angle $\theta>0$, and let $e$ be an edge (resp. face) of $\K$. Then for every $p>2$
and for every integer $k_{max}$
there exists a constant $C_{p,\theta,kmax}$  such that
\begin{equation}\label{propA:D}
\int_e\vect{v}\cdot(\taub\cdot\n)\ds\le C_{p,\theta,k_{max}}\, h_{\K}^{-1/2}\|\vect{v}\|_{0,e}\,
\Big(h_{\K}\|{\rm \vect{div}}\taub\|_{0,\K}\,+\,h_{\K}^{\frac{d(p-2)}{2p}}\|\taub\|_{0,p,\K}\Big)
\end{equation}
for every $\taub\in(L^p(\O))^{d\times d}_{sym}$ having divergence in
$\vect{L}^2$ and for every $\vect{v} \in\Pb^{k_{max}}(\K)$.
\end{proposition}

\proof
First we go to the reference element $\hat{\K}$:
\begin{equation}\label{torefTp3}
\Big|\int_e\vect{{v}}\cdot({\taub}\cdot\n)\,\ds\le C_{\theta}|e|\Big|\int_{\hat{e}}\vect{\hat{v}}\cdot(\hat{\taub}\cdot\hat{\n})\,d\hat{s}\Big|
\le C_{\theta}h_e^{d-1}\Big|\int_{\hat{e}}\vect{\hat{v}}\cdot(\hat{\taub}\cdot\hat{\n})\,d\hat{s}\Big|
\end{equation}
where $\hat{\vect{v}}$ and $\hat{\taub}$ are the usual covariant and
contra-variant images of $\vect{v}$ and $\taub$, respectively. And,
here and throughout his proof, the constants $C_{\theta}$ and
$C_{\theta,k_{max}}$ may assume different values at different
occurrences.  Note that $\hat{\vect{v}}$ will still be a vector-valued
polynomial of degree $\le k_{max}$ and the space $H(\div, \K)$ is
effectively mapped into $H(\div, \hat{\K})$ by means of the
contra-variant mapping. Then for every component $\hat{v}$ of
$\hat{\vect{v}}$, we construct the auxiliary function $\varphi_v$ as
follows. First we define $\varphi_v$ on $\partial\hat{\K}$ by setting
it as equal to $\hat{v}$ on $\hat{e}$ and zero on the rest of
$\partial\hat{\K}$. Then we define $\varphi_v$ in the interior using
the harmonic extension.  It is clear that $\varphi_v$ will belong to
$W^{1,p'}(\hat{\K})$ (remember that $p>2$ so that its conjugate index
$p'$ will be smaller than $2$). Using the fact that $\hat{\vect{v}}$
is a polynomial of degree $\le k_{max}$, it is not difficult to see
that
\begin{equation}
\|\vect{\varphi}_v\|_{W^{1,p'}(\hat{\K})}
\le\,\hat{C}_{\theta,k_{max}}\|\hat{\vect{v}}\|_{0,\hat{e}}.
\end{equation}
Integration by parts then gives
\begin{equation}\label{partsrefTp3}
\begin{aligned}
\int_{\hat{e}}\hat{\vect{v}}\cdot(\hat{\taub}\cdot\hat{\n})\,d\hat{s}&=
\int_{\partial\hat{\K}}\vect{\varphi}_{v}\cdot(\hat{\taub}\cdot\hat{\n})\,d\hat{s}\\
&=
\int_{\hat{\K}}\vect{\nabla}\vect{\varphi}_{v}:\hat{\taub}\,d\hat{x}
-\int_{\hat{\K}}\vect{\varphi}_{v}\cdot{\rm \vect{div}}{\hat{\taub}}\,d\hat{x}\\
&\le
|\vect{\varphi}_v|_{W^{1,p'}(\hat{\K})}\|\hat{\taub}\|_{(L^p(\hat{\K}))^{d\times
    d}_{sym}}+
\|\vect{\varphi}_v\|_{0,\hat{\K}}\|{\rm \vect{div}}{\hat{\taub}}\|_{0,\hat{\K}}\\
&\le
\hat{C}\Big(\|\hat{\vect{v}}\|_{0,\hat{e}}\,\|\hat{\taub}\|_{(L^p(\hat{\K}))^{d\times
    d}_{sym}}\,+\,
\|\vect{\varphi}_v\|_{0,\hat{e}}\|{\rm \vect{div}}{\hat{\taub}}\|_{0,\hat{\K}}\Big)\\
&\le
\hat{C}\,\|\hat{\vect{v}}\|_{0,\hat{e}}\,\Big(\|\hat{\taub}\|_{(L^p(\hat{\K}))^{d\times
    d}_{sym}}\,+\|{\rm \vect{div}}{\hat{\taub}}\|_{0,\hat{\K}}\Big).
    \end{aligned}
\end{equation}
Then we recall the inverse transformations (from $\hat{\K}$ to $\K$):
\begin{eqnarray*}
&& \|\hat{\vect{v}}\|_{0,\hat{e}}\le C_{\theta}h_e^{-\frac{d-1}{2}}\|\vect{v}\|_{0,{e}},\quad
\|\hat{\taub}\|_{(L^p(\hat{\K}))^{d\times d}_{sym}}\le C_{\theta}
h_{\K}^{-\frac{d}{p}}\|{\taub}\|_{(L^p({\K}))^{d\times d}_{sym}},\\
&&\|{\rm \vect{div}}\hat{\taub}\|_{0,\hat{\K}}\le C_{\theta}h_{\K}^{\frac{2-d}{2}}
\|{\rm \vect{div}\taub}\|_{0,\K}.
\end{eqnarray*}
Inserting this into \eqref{partsrefTp3} and then in \eqref{torefTp3} we have then
\begin{equation*}
\int_e\vect{v}\cdot(\taub\cdot\n)\ds\le C_{p,\theta,k_{max}}\,
h_e^{d-1}\,h_e^{-\frac{d-1}{2}}\|\vect{v}\|_{0,{e}}\Big(
h_{\K}^{-\frac{d}{p}}\|{\taub}\|_{(L^p({\K}))^{d\times d}_{sym}}+h_{\K}^{\frac{2-d}{2}}
\|{\rm \vect{div}\taub}\|_{0,\K}
\Big).
\end{equation*}
Now we note that
\begin{equation*}
-\frac{1}{2}+\frac{d(p-2)}{2p}=d-1-\frac{d-1}{2}-\frac{d}{p},
\end{equation*}
and that
\begin{equation*}
-\frac{1}{2}+1=d-1-\frac{d-1}{2}+\frac{2-d}{2},
\end{equation*}
and the proof then follows immediately.
\endproof

With this result in hand, we can show the Korn inequality
\eqref{korn1} given in Lemma \ref{korn} for $d=3$. It is necessary to modify the proof
in only  two places: the definition of the space of rigid
motions on $\O$, $\vect{RM}(\O)$, and the application of Proposition
\ref{propA}. The space $\vect{RM}(\O)$ is now defined by:
\begin{equation*}
\vect{RM}(\O)=\left\{ \, \vect{a}+\vect{b} \vect{x}\,\,:\,\, \vect{a}\in \re^{d} \quad \vect{b} \in so(d) \,\,\right\}
\end{equation*}
with $so(d)$ denoting the  space of the skew-symmetric $d\times d$ matrices.

To prove \eqref{boundtoprove} (and so conclude the proof of \eqref{korn1}), estimate  \eqref{usapropA} is replaced by  estimate  \eqref{usapropA:D} below, which is obtained as follows: first, by applying \eqref{propA:D} (instead of \eqref{propA})  from Proposition \ref{propA:D} to each $e$ in the last term in \eqref{byparts} and then by using
the generalized H\"older inequality with the same exponents as for $d=2$ (with $q=1/2$ and $r=2p/(p-2)$, so that $\frac{1}{p}+\frac{1}{q}+\frac{1}{r}=1$)

\begin{eqnarray}
\sum_{e\in \Eho}\int_e \jump{\bv_t} : \av{ \vect{\tau}} & \le &
C_{p,\theta,kmax}\, \sum_{\K\in\Th}\sum_{e\in\partial\K}
 h_{\K}^{-1/2}\|\jump{\vect{v}_t}\|_{0,e}\,h_{\K}\|{\rm\bf div}\vect{\tau}\|_{0,\K}
\nonumber \\
&& + C_{p,\theta,kmax}\, \sum_{\K\in\Th}\sum_{e\in\partial\K}
h_{\K}^{-1/2}\|\jump{\vect{v}_t}\|_{0,e}\|\,h_{\K}^{\frac{d(p-2)}{2p}}\|\taub\|_{0,p,\K}
 \nonumber\\
& \le &  C h\,|\jump{\vect{v}_t}|_{\ast}\,\|{\rm\bf div}\taub\|_{0,\O}\label{usapropA:D} \\
&&  +
 C\Big(\sum_{e\in\Eho}h_e^{-1}|\jump{\vect{v}_t}|^2_{0,e}\Big)^{1/2}
 \Big(\sum_{e\in\Eho}\|\taub\|_{0,p,\K(e)}^p\Big)^{1/p}
 \Big(\sum_{e\in\Eho}h_e^{\frac{d(p-2)}{2p}
   r}\Big)^{1/r}\nonumber \\
 &\le&
C |\jump{\vect{v}_t}|_{\ast}\,h\,\|{\rm\bf div}\taub\|_{0,\O}+C\,|\jump{\vect{v}_t}|_{\ast}\,\|\taub\|_{0,p,\O}\,\mu({\O})^{1/r}\nonumber
\end{eqnarray}
Here, as in estimate \eqref{usapropA}, $\mu(\O)$
denotes the measure of the domain $\O$, and the constant
$C$ still depends on $p$, $k_{max}$, and on the maximum angle in the
decomposition $\Th$.  The rest of the proof of Lemma \ref{korn}
proceeds as for $d=2$.

\section*{Acknowledgments}
{The authors thank one of the referees for helpful comments on the first version of this work.}
Part of this work was completed while the first author was visiting
IMATI-CNR of Pavia. She is grateful to the IMATI for the kind
hospitality.  The first author was partially supported by MINECO through
grant MTM2011-27739-C04-04.  The second and third authors were
partially supported by the Italian MIUR through the project PRIN2008.
The last two authors were partially supported by National Science
Foundation grant DMS-1217142 and US Department of Energy grant
DE-SC0009603. The authors also thank Feiteng Huang for the
  help with the numerical tests and in particular for putting the
  discretization within the FEniCS framework, {and to Harbir Antil 
for pointing out references \cite{beirao2004, amrouche2011}.}

\bibliographystyle{abbrv}
\bibliography{biblio_only_cited}
\end{document}